\documentclass[10pt,reqno]{amsart}

\usepackage{amsmath,amscd,amssymb,amsthm}
\usepackage{tikz-cd}
\usepackage{comment}

\makeatletter
\def\th@plain{%
  \thm@notefont{}% same as heading font
  \itshape % body font
}
\def\th@definition{% 
  \thm@notefont{}% same as heading font
  \normalfont % body font
}
\makeatother

\usepackage{enumitem}
\setlist[enumerate]{label=(\roman*),leftmargin=0.8cm}
\usepackage{todonotes}
\usepackage{color}

\newtheorem{proposition}{Proposition}[section]
\newtheorem{lemma}[proposition]{Lemma}
\newtheorem{theorem}[proposition]{Theorem}
\newtheorem{corollary}[proposition]{Corollary}

\theoremstyle{definition}
\newtheorem{remark}[proposition]{Remark}
\newtheorem{definition}[proposition]{Definition}

\numberwithin{equation}{section} 

\DeclareMathOperator{\Id}{Id}
\DeclareMathOperator{\tr}{tr}

\DeclareMathOperator{\Aut}{Aut}
\DeclareMathOperator{\Ric}{Ric}

\DeclareMathOperator{\V}{V}
\DeclareMathOperator{\Ima}{Im}
\DeclareMathOperator{\Isom}{Isom}

\DeclareMathOperator{\Diff}{Diff}
\DeclareMathOperator{\Lie}{Lie}

\renewcommand{\L}{\mathcal{L}}
\renewcommand{\phi}{\varphi}

\DeclareMathOperator{\isom}{isom}

\newcommand{\N}{\mathbb{N}}
\newcommand{\R}{\mathbb{R}}
\newcommand{\C}{\mathbb{C}}

\newcommand{\pr}{\mathbb{P}}
\renewcommand{\epsilon}{\varepsilon}
\renewcommand{\H}{\mathcal{H}}

\newcommand{\M}{\mathcal{M}}
\newcommand{\scN}{\mathcal{N}}
\newcommand{\scO}{\mathcal{O}}

\newcommand{\scK}{\mathcal{K}}

\newcommand{\ddb}{i\partial \bar\partial}
\newcommand{\scL}{\mathcal{L}}

\newcommand{\mfh}{\mathfrak{h}}

\newcommand{\mfg}{\mathfrak{g}}

\newcommand{\scM}{\mathcal{M}}

\newcommand{\mfk}{\mathfrak{k}}
\newcommand{\mfa}{\mathfrak{a}}

\newcommand{\scD}{\mathcal{D}}
\renewcommand{\V}{\mathcal{V}}
\newcommand{\scV}{\mathcal{V}}

\newcommand{\scR}{\mathcal{R}}

\newcommand{\scH}{\mathcal{H}}
\newcommand{\ddc}{i\partial\bar\partial}

\newcommand{\ddbar}{\partial\bar{\partial}}

\title[Uniqueness of optimal symplectic connections]{Uniqueness of optimal symplectic connections}

\author[Ruadha\'i Dervan and Lars Martin Sektnan]{Ruadha\'i Dervan and Lars Martin Sektnan}
\address{Ruadha\'i Dervan, DPMMS, Centre for Mathematical Sciences, Wilberforce Road, Cambridge CB3 0WB, United Kingdom}
\email{R.Dervan@dpmms.cam.ac.uk}

\address{Lars Martin Sektnan, Institut for Matematik, Aarhus University, 8000, Aarhus C, Denmark}
\email{lms@math.au.dk}

\begin{document}

\begin{abstract} Consider a holomorphic submersion between compact K\"ahler manifolds, such that each fibre admits a constant scalar curvature K\"ahler metric. When the fibres admit continuous automorphisms, a choice of fibrewise constant scalar curvature K\"ahler metric is not unique. An optimal symplectic connection is choice of fibrewise constant scalar curvature K\"ahler metric satisfying a geometric partial differential equation. The condition generalises the Hermite-Einstein condition for a holomorphic vector bundle, through the induced fibrewise Fubini-Study metric on the associated projectivisation. 

We prove various foundational analytic results concerning optimal symplectic connections. Our main result proves that optimal symplectic connections are unique, up to the action of the automorphism group of the submersion, when they exist. Thus optimal symplectic connections are \emph{canonical} relatively K\"ahler metrics when they exist. In addition, we show that the existence of an optimal symplectic connection forces the automorphism group of the submersion to be reductive, and that an optimal symplectic connection is automatically invariant under a maximal compact subgroup of this automorphism group. We also prove that when a submersion admits an optimal symplectic connection, it achieves the absolute minimum of a natural log norm functional which we define.
\end{abstract}

\maketitle

\section{Introduction}

Consider a holomorphic submersion $\pi: X \to B$ between compact complex manifolds, with $\alpha$ a relatively K\"ahler class on $X$ and $\beta$ a K\"ahler class on $B$. The question that motivates the present work is: what does it mean for $\omega_X \in \alpha$ to be a \emph{canonical} relatively K\"ahler metric? This question has well-known answers in the following situations:
\begin{enumerate}
\item $B$ is a point, so that $(X,\alpha)$ is a compact K\"ahler manifold. In this case, a \emph{constant scalar curvature K\"ahler} (cscK)   metric $\omega_X \in \alpha$ is the natural choice of canonical metric. K\"ahler-Einstein metrics are a special case.
\item All fibres $X_b$ for $b\in B$ have discrete automorphism group. In this case, on each fibre a cscK metric is unique, so a natural choice is a form $\omega_X \in \alpha$ which restricts to the cscK metric on each fibre $X_b$.
\item $X=\pr(E)$ and $\alpha = c_1(\scO_{\pr(E)}(1))$. Then a Hermite-Einstein metric $h$ on $E$ induces a  form $\omega_X \in \alpha$  which restricts to a Fubini-Study metric on each fibre, which is therefore also cscK on each fibre. Then this $\omega_X$ is a natural choice of relatively K\"ahler metric.
\end{enumerate}

While one has natural answers in the above situations, they are rather sparse examples. Most submersions of interest in higher dimensions are certainly not of the above form.

In each of the above cases, one obtains a uniqueness statement: the canonical choice of metric is actually unique up a natural class of automorphisms. In the first, uniqueness is up to the action of $\Aut(X)$ \cite{donaldson-scalar-curvature, berman-berndtsson}. In the second, one obtains uniqueness up to pullback of a form from $B$. In the third, one has uniqueness up to the action of the global endomorphisms of $E$ \cite{donaldson-hermite-einstein}. These uniqueness statements, which are the best possible, are really what is meant by having a canonical metric. One also sees from the projective bundle situation that, in order to understand the geometry of the submersion, it is natural to fix a K\"ahler metric on the base $B$, and so we fix one throughout.

In previous work, we introduced a candidate answer for our motivating question, in the form of an \emph{optimal symplectic connection} \cite{morefibrations}. This is a relatively K\"ahler metric $\omega_X \in \alpha$, which restricts to a cscK metric on each fibre, and satisfies an additional geometric partial differential equation described explicitly in Section \ref{sec:optimal}. Briefly, the equation is a fully-nonlinear second-order elliptic PDE on a vector bundle parametrising the fibrewise holomorphic vector fields; the PDE involves the symplectic curvature of the form $\omega_X$ together with a relative version of the Ricci curvature. The condition generalises the Hermite-Einstein equation when $X=\pr(E)$, and as we show in Section \ref{sec:optimal}, also simplifies to a generalised Hermite-Einstein type equation when all fibres $X_b$ are K\"ahler-Einstein Fano manifolds. 

Note that when the fibres admit continuous automorphisms, there is an infinite dimensional family of relatively K\"ahler metrics which are cscK on each fibre. We conjectured that solutions of the optimal symplectic connection equation are \emph{unique}, meaning that optimal symplectic connections do give a \emph{canonical} choice of $\omega_X \in \alpha$ when they exist \cite[Conjecture 1.2]{morefibrations}. Here we prove that conjecture.

\begin{theorem}\label{intro-uniqueness} Suppose $\omega_X, \omega_X' \in \alpha$ are two optimal symplectic connections. Then there is a $g \in \Aut_0(\pi)$ and a function $\phi_B \in C^{\infty}(B,\R)$ such that $$\omega_X = g^*\omega_X' + \pi^*(\ddb \phi_B).$$
\end{theorem}

\noindent Here $\Aut_0(\pi)$ denotes biholomorphisms of $X$ preserving $\pi$. This is the best possible uniqueness result, and implies that optimal symplectic connections do give a \emph{canonical} choice of relatively K\"ahler metric on submersions, when they exist. The result generalises and recovers Donaldson's uniqueness of Hermite-Einstein metrics \cite{donaldson-hermite-einstein}, with a completely different method.

Analytic objects, arising from differential geometry, that one can \emph{uniquely} associate to holomorphic objects are often extremely useful: a typical example is the use of constant curvature metrics in the study of the moduli space of compact Riemann surfaces, which is essentially Teichm\"uller theory. Uniqueness statements also  play a crucial role in the analytic approach to forming moduli of polarised \emph{manifolds} admitting canonical metrics in higher dimensions \cite{moduli, inoue}, and so it is natural to expect our uniqueness result to be useful in forming moduli of \emph{submersions} over a fixed base (compare also the use of Hermite-Einstein metrics in the study of moduli of holomorphic vector bundles). 

We also prove the following results that demonstrate how optimal symplectic connections reflect the geometry of submersions.

\begin{theorem}\label{intro:matsushima}
Suppose $\pi: (X,\alpha) \to (B,\beta)$ admits an optimal symplectic connection $\omega_X$. Then
\begin{enumerate}
\item the Lie algebra of holomorphic vector fields preserving $\pi$ and vanishing somewhere is reductive;
\item the isometry group $\Isom_0(\pi,\omega_X)$ is a maximal compact subgroup of $\Aut_0(\pi)$.
\end{enumerate}
\end{theorem}

\noindent These are analogues of foundational results for cscK metrics due to Matsushima, Lichnerowicz and Calabi. More precise statements can be found in Section \ref{sec:energy}.

Our final main result concerns a natural log norm functional $\scN$, a real-valued functional defined on the space of relatively K\"ahler metrics which are cscK on each fibre. When $\pi: X \to B$ admits automorphisms, we restrict to relatively K\"ahler metrics invariant under a maximal compact subgroup of $\Aut_0(\pi)$; if $\pi: (X, \alpha) \to (B, \beta)$ admits an optimal symplectic connection $\omega_X$, we take this maximal compact subgroup to be $\Isom_0(\pi,\omega_X)$. Our functional $\scN$ is the analogue of the Mabuchi functional, which detects the existence of cscK metrics, and the Donaldson functional, which determines the existence of Hermite-Einstein metrics.

\begin{theorem}\label{thm:functional} Suppose $\pi: (X, \alpha) \to (B, \beta)$ admits an optimal symplectic connection $\omega_X$. Then the absolute minimum of $\scN$ is achieved by $\scN(\omega_X),$ and hence $\scN$ is bounded below. \end{theorem}

\noindent We conjecture that the existence of an optimal symplectic connection is equivalent to a stronger notion of \emph{properness} of the functional $\scN$ on the space of relatively K\"ahler metrics which are cscK on each fibre.

Focusing for the moment on our uniqueness results, there are two strategies to proving uniqueness of classes of metrics in K\"ahler geometry. The first relies on \emph{convexity properties} for infinite dimensional log norm functionals; this was Donaldson's approach to proving uniqueness of Hermite-Einstein metrics \cite{donaldson-hermite-einstein}, and is the approach used by Berman-Berndtsson to establish uniqueness of cscK metrics \cite{berman-berndtsson} (following Donaldson's programme \cite{donaldson-uniqueness-programme}). The second, which only applies when the K\"ahler manifolds are projective, is to perturb to an easier \emph{finite dimensional} problem; this approach was used by Donaldson to prove uniqueness of cscK metrics \cite{donaldson-scalar-curvature}.

We use a different approach that blends the two ideas, by perturbing to another infinite dimensional problem where uniqueness is already known. In prior work, we explained how to construct extremal metrics on the total space of submersions using a canonical twisted extremal metric on $B$ and optimal symplectic connections \cite{morefibrations}. Such twisted extremal metrics do not always exist on $B$, and moreover morally one should not use canonical objects \emph{on} $B$ to study the geometry of submersions \emph{over} $B$; the metric chosen on $B$ should, in some sense, be irrelevant. 

Here, given \emph{any} form $\omega_B$ on $B$, we use that $\omega_B$ can always be seen as a twisted cscK metric for an appropriate twist to constuct \emph{twisted} extremal metrics on $X$ itself, in the class $k\beta + \alpha$. This requires developing some novel techniques in the study of adiabatic limit problems, which in K\"ahler geometry originated with work of Hong and Fine \cite{hong, fine1}. In particular, we show that the twisted extremal metrics we produce are sufficiently well approximated by the approximate solutions we construct that one can pass from statements about twisted extremal metrics, to statements about each term in our approximate solution. Thus in some sense, we perturb from uniqueness of twisted extremal metrics to uniqueness of optimal symplectic connections. The precise statement is the following.

\begin{theorem}\label{intro-adiabatic} Suppose $\omega_X$ is an optimal symplectic connection. Then there is a K\"ahler metric $\xi$ on $B$ such that $k\beta+\alpha$ admits a twisted extremal metric with twist $\pi^*\xi$ for all $k \gg 0$.
 \end{theorem}

Twisted extremal metrics are best viewed as canonical metrics on \emph{maps} between complex manifolds \cite{fibrations, stablemaps}; in our case, the map is $\pi: X \to B$. Uniqueness statements for twisted extremal metrics are proved in \cite{keller, berman-berndtsson,fibrations}, and when $\Aut_0(\pi)$ is the identity, combined with Theorem \ref{intro-adiabatic} and our new adiabatic limit techniques are enough to prove uniqueness of optimal symplectic connections. In the case $\Aut_0(\pi)$ is non-trivial, we have to work harder, and employ techniques concerning the action of the automorphism group on the space of K\"ahler potentials developed in the important work of Darvas-Rubinstein \cite{darvas-rubinstein}.

Theorem \ref{intro-uniqueness} proves uniqueness of optimal symplectic connections, but not existence. We conjectured in \cite{morefibrations, stablefibrations} that existence is equivalent to a notion of \emph{stability} of algebro-geometric fibrations. This also motivates our conjecture that one should be able to form a moduli space of submersions over a fixed base $B$ which admit an optimal symplectic connection; as mentioned above, uniqueness results are crucial in the analytic approach to such questions. We also remark that our algebro-geometric conjecture would imply that the \emph{existence} of solutions is actually independent of $\omega_B$ chosen; as one sees from the Hermite-Einstein situation, this has no relevance for uniqueness questions, as one obtains uniqueness of $\omega_X$ after choosing $\omega_B$. 

Finally, we remark that Theorem \ref{intro-adiabatic} provides the following purely algebro-geometric statement:

\begin{corollary} 

Suppose $E$ is a stable vector bundle over $(B, L)$, where $L$ is ample on $B$. Then for all $k \gg j \gg 0$, the map $$\pi: (\pr(E), kL+ \scO_{\pr(E)}(1)) \to (B,j L)$$ is a K-semistable map.

\end{corollary}

\noindent Here K-semistability of maps is meant in the sense of \cite{stablemaps}. This provides a strong, and perhaps surprising, link between stability of bundles and K-stability of maps. The above follows from Theorem \ref{intro-adiabatic} by using that the existence of a twisted cscK metric implies K-semistability of the map $\pi$ \cite{uniform, stablemaps}.

\vspace{4mm} \noindent {\bf Outline:} Section \ref{sec:prelims} contains material on cscK metrics, optimal symplectic connections and twisted extremal metrics that will be essential later. The only new material is a simplification of the optimal symplectic connection condition for Fano submersions, in Proposition \ref{fano-simplification}. Section \ref{sec:adiabatic} proves Theorem \ref{intro-adiabatic}, together with some crucial estimates bounding the twisted extremal and approximately twisted extremal metrics constructed there. We prove our main results, Theorems \ref{intro-uniqueness} \ref{intro:matsushima} and \ref{thm:functional}, in Section \ref{sec:uniqueness}.

\vspace{4mm} \noindent {\bf Acknowledgements:} We thank the referee for their helpful comments and corrections. LMS's postdoctoral position is supported by Villum Fonden, grant 0019098.

\section{Preliminaries}\label{sec:prelims}
\subsection{Constant scalar curvature K\"ahler metrics}

Let $X$ be an $n$-dimensional compact K\"ahler manifold and let $\alpha$ be a K\"ahler class on $X$. The \emph{scalar curvature} of $X$ is the contraction $$S(\omega) = \Lambda_{\omega} \Ric\omega$$ of the \emph{Ricci curvature} $$\Ric\omega = -i\ddbar \log \omega^n.$$ Thus $$S(\omega)\omega^n = n \Ric\omega \wedge \omega^{n-1}.$$

\begin{definition}

We say that $\omega$ is a \emph{constant scalar curvature K\"ahler metric} (cscK) if $S(\omega)$ is constant.

\end{definition}

CscK metrics give a canonical choice of K\"ahler metric, when they exist, in the following sense. Denote by $\Aut_0(X)$ the connected component of the identity inside the group of biholomorphisms $\Aut(X)$ of $X$.

\begin{theorem}\label{csck-uniqueness}\cite{donaldson-scalar-curvature, berman-berndtsson} Suppose $\omega, \omega' \in \alpha$ are cscK metrics. Then there is a $g \in \Aut_0(X)$ with $g^*\omega = \omega'$. \end{theorem}

Thus there is a close relationship between cscK metrics and automorphisms. Going further, denote by $$\mfh \subset H^0(X,TX^{1,0})$$ the Lie algebra of holomorphic vector fields which vanish somewhere. Via the natural isomorphism $TX^{1,0} \cong TX $, holomorphic vector fields correspond to \emph{real holomorphic vector fields}: vector fields whose flows preserve the complex structure. Denote by $\mfk \subset \mfh$ the Lie subalgebra of vector fields which correspond to real holomorphic vector fields that are Killing. 

\begin{theorem}\label{reductivity}\cite[Theorem 3.5.1, Theorem 3.6.1]{gauduchon} Suppose $\omega \in \alpha$ is a cscK metric. Then
\begin{enumerate}
\item $\mfh = \mfk \oplus J\mfk$, so $\mfh$ is reductive;
\item $\Isom_0(X,\omega)$ is a maximal compact subgroup of $\Aut_0(X)$.
\end{enumerate}
\end{theorem}

\noindent In the case $X$ is projective and $\alpha = c_1(L)$, then $\mfh$ can be characterised as the Lie algebra of $\Aut(X,L)$, the automorphisms which linearise to $L$. Thus in the projective setting, the first item of the above states that this Lie group is reductive.

The final result we require is of a more analytic nature. For $\phi \in C^{2}(X,\C)$, denote by \begin{equation}\label{eq:curlyd}\scD\phi = \bar\partial\nabla^{1,0}\phi\end{equation} with $\nabla^{1,0}$ the $(1,0)$-part of the gradient of $\phi$ with respect to $\omega$. Let $\scD^*$ be the $L^2$-adjoint of $\scD$ with respect to the $L^2$-inner product $$\langle \phi,\psi\rangle = \int_X \phi\bar \psi \omega^n.$$ 

\begin{theorem} Suppose $\omega$ is a cscK metric. The operator $$\scD^*\scD: C^{k,\alpha}(X,\C) \to C^{k-4,\alpha}(X,\C)$$ with $k \geq 4$ is a real fourth-order elliptic operator, hence invertible orthogonal to its kernel. Its kernel satisfies $$\ker \scD^*\scD =  \ker \scD \cong \mfg,$$ via $$f \to \nabla^{1,0}f.$$ \end{theorem}

\noindent To say that $\scD^*\scD$ is a real operator means that it sends real functions to real functions. If $\omega$ is not cscK, it is no longer even true that $\scD^*\scD$ is a real operator. Elements of the kernel of $\scD$ are called holomorphy potentials, as if $\phi \in \ker \scD$, then $\nabla^{1,0}\phi \in \mfg$ is a holomorphic vector field.

We end the section with the following elementary Lemma (see e.g. \cite[Lemma 4.10]{gabor-book}), describing the dependency of the holomorphy potential on the metric within a fixed class.

\begin{lemma}\label{lem:changeinpot} If $\nu$ is a holomorphic vector field with potential $h$ with respect to $\omega$, then $h + \nu (\phi)$ is a holomorphy potential with respect to $\omega + \ddb \phi$. 
\end{lemma}

\subsection{Optimal symplectic connections}\label{sec:optimal}
Consider now a holomorphic submersion $\pi: X\to B$ between compact K\"ahler manifolds, with $B$ of dimension $n$ and the fibres of dimension $m$. In order to discuss metric properties of the submersion, we recall that a \emph{relative K\"ahler class} $\alpha$ on $X$ is an element $\alpha \in H^2(X,\R)$ such that its restriction  to  each fibre $X_b = \pi^{-1}(b)$, which we write as $\alpha_b$, is a K\"ahler class. Thus is $\beta$ if K\"ahler on $B$, then $k\beta+\alpha$ is K\"ahler for all $k \gg 0$. 

Suppose $\omega_X \in \alpha$ is a relatively K\"ahler metric, so that $\omega$ is closed and $\omega_b = \omega|_{ X_b}$ is K\"ahler for all $b \in B$. We say that $\omega_X$ is \emph{relatively cscK} if $\omega_b$ is cscK for all $b \in B$. When all fibres $(X_b,\alpha_b)$ admit a cscK metric, one can show that the class $\alpha$ admits a relatively cscK metric \cite[Lemma 3.8]{morefibrations}. It is easy to see from Theorem \ref{csck-uniqueness} that when $\Aut(X_b)$ is discrete for all $b \in B$, such a relatively cscK metric is actually unique. However, when the fibres $X_b$ have continuous automorphisms, absolute uniqueness no longer holds, and it is natural to ask if there is a canonical choice of relatively cscK metric.

 Much as a canonical choice of K\"ahler metric is given by a solution to a geometric partial differential equation, our answer to this question will be phrased analytically. Thus it will be necessary to assume a sort of smoothness hypothesis, namely that $\dim \mfh_{b}$ is independent of $b \in B$, with $\mfh_b$ the Lie algebra of holomorphic vector fields which vanish somewhere on $X_b$.

To each fibre $(X_b,\alpha_b)$ one can associate a real vector space $E_b \subset C^{\infty}_0(X_b,\R)$ by setting $E_b = \ker \scD_b$, with $ \scD_b$ the operator defined via Equation \eqref{eq:curlyd} and $C^{\infty}_0(X_b,\R)$ denoting functions of integral zero with respect to $\omega_b^n$. Thus $E_b$ can be thought of as parametrising holomorphic vector fields on $(X_b,\alpha_b)$. Going further, one can naturally define a smooth vector bundle $E \to B$ associated to the submersion $\pi: X\to B$ with fibre $E_b$ \cite[Section 3.1]{morefibrations}; this is the step at which our assumption that $\dim \mfh_{b}$ is independent of $b \in B$ enters. Thus a section of $E$ over $B$ corresponds to a function on $X$ whose restriction to each fibre $X_b$ is a mean-value zero holomorphy potential with respect to $\omega_b$. We will denote $C^{\infty}_E(X,\R)$ the space of global sections of $E$; note that one has a natural inclusion $$C^{\infty}_E(X, \R) \subset C^{\infty}(X,\R).$$

This also defines a natural splitting of the space $C^{\infty}(X,\R)$ \cite[Section 3.1]{morefibrations} into three components; one is $C^{\infty}_E(X, \R)$ described above. The map $$C^{\infty}(X_b,\R) \to \R, \qquad \phi \to \int_{X_b}\phi \omega_b^n$$ induces a projection $$C^{\infty}(X,\R) \to C^{\infty}(B,\R), \qquad \phi \to \int_{X/B}\phi \omega_X^n,$$ where by definition the fibre integral is defined by $\left(\int_{X/B}\phi \omega_X^n\right)(b) = \int_{X_b}\phi \omega_b^n$. 

The third component is defined as follows. On each fibre there is a natural $L^2$-inner product on functions. Denoting $C^{\infty}_R(X_b,\R)$ the $L^2$-orthogonal complement of $E_b \subset C^{\infty}_0(X_b,\R)$,  we defie a space $C^{\infty}_R(X,\R) \subset C^{\infty}(X,\R)$ to be functions whose restriction to $X_b$ lies in $C^{\infty}_R(X_b,\R)$. This produces the desired decomposition $$C^{\infty}(X,\R) = C^{\infty}(B)\oplus C_E^{\infty}(X)\oplus C_R^{\infty}(X),$$ where we have omitted the fact that these are real valued functions in the notation. It will be useful to denote the natural projection of functions onto the $C_E^{\infty}(X)$-component by $$p: C^{\infty}(X,\R) \to C^{\infty}_E(X).$$

We now return to the equation defining a canonical choice of relatively K\"ahler metric. The partial differential equation is called the \emph{optimal symplectic connection} equation, introduced by the authors  \cite[Section 3]{morefibrations}. This equation can be viewed as an elliptic partial differential equation on the bundle $E$. The condition is phrased in terms of curvature quantities associated to $\omega_X$, which we now briefly recall. We refer to \cite[Section 3]{morefibrations} for further details and basic results.

The relatively K\"ahler metric $\omega_X$ defines a hermitian metric on the relative tangent bundle, which is a holomorphic vector bundle of rank $m$ by the hypothesis that $\pi: X \to B$ is a holomorphic submersion. Taking the top exterior power, $\omega_X$ therefore induces a hermitian metric on the relative anticanonical class $-K_{X/B}$, with curvature which we denote $\rho \in c_1(-K_{X/B})$.

The form $\omega_X$ induces a smooth splitting of vector bundles $$TX \cong \scH \oplus \scV,$$ with $\scV = \ker d\pi$ the vertical tangent bundle and $\scH$ the $\omega_X$-orthogonal complement of $\scV$; in this context, $\omega_X$ is usually called a symplectic connection. This further induces a splitting on all tensors. Thus $\omega_X$ induces an Ehresmann connection, which has curvature $F_{\scH}$, a two-form on $B$ with values in fibrewise Hamiltonian vector fields. Let $\mu^*$ denote the fibrewise co-moment map, sending on each fibre a Hamiltonian vector field to its integral zero assocated Hamiltonian function. $\mu^*$ thus allows us to construct $\mu^*F_{\scH}$, a two-form on $B$ with values in fibrewise Hamiltonian functions. The ``minimal coupling'' equation states that \begin{equation}\label{eqn:minimal-coupling}(\omega_X)_{\scH} = \mu^*F_{\scH} + \pi^*\eta,\end{equation} with $(\omega_X)_{\scH}$ the purely horizontal component of $\omega_X$ and $\eta $ two-form on $B$. 

Let $\omega_B \in \beta$ be a K\"ahler metric on $B$. Then $\omega_B$ induces a contraction operator $\Lambda_{\omega_B}$ on purely horizontal forms, and thus for example $\Lambda_{\omega_B}\rho_{\scH}$ is naturally a function on $X$.

On each fibre $\omega_b$ induces a Laplacian operator on functions $\Delta_b$; these glue to form the vertical Laplacian operator $\Delta_{\scV}$ defined by $$\Delta_{\scV}\phi|_{X_b} = \Delta_b (\phi|_{X_b}).$$

\begin{definition}\cite{morefibrations} We say that a relatively cscK metric $\omega_X$ is an \emph{optimal symplectic connection} if $$p(\Delta_{\scV} (\Lambda_{\omega_B}\mu^*F_{\scH}) + \Lambda_{\omega_B} \rho_{\scH}) = 0.$$\end{definition}

 We showed in \cite[Proposition 3.17]{morefibrations} that when $X=\pr(E)$, the optimal symplectic connection condition reduces to the Hermite-Einstein condition. The fibre in that situation, namely projective space, is the simplest example of a K\"ahler-Einstein Fano manifold. We call we call $\pi: (X,\alpha) \to (B, \beta)$ a \emph{Fano submersion} if each fibre $X_b$ is a Fano manifold, with $\alpha = c_1(-K_{X/B})$ the relative anti-canonical class. The optimal symplectic connection also simplifies considerably for Fano submersions:

\begin{proposition}\label{fano-simplification} Suppose $(X,\alpha)$ is a Fano submersion. Then $\omega_X \in \alpha$ is an optimal symplectic connection if and only if $$p(\Lambda_{\omega_B} \mu^*F_{\scH}) = 0.$$\end{proposition}

\begin{proof}

It is shown in \cite[Proposition 3.10]{morefibrations} that in this case $\rho = \omega_X + \pi^*\xi$ for some form $\xi$ on $B$, and so the equation simplifies to $$p(\Delta_{\scV} \Lambda_{\omega_B}\mu^*F_{\scH} + \Lambda_{\omega_B}\mu^*F_{\scH}) = 0$$ by Equation \eqref{eqn:minimal-coupling}.

We claim that $$p(\Delta_{\scV} \Lambda_{\omega_B}\mu^*F_{\scH})  = p(\Lambda_{\omega_B}\mu^*F_{\scH}),$$ which will prove the desired result. Since the projection $p: C^{\infty}(X,\R) \to C^{\infty}_E(X, \R)$ is the orthogonal projection onto fibrewise holomorphy potentials, by non-degeneracy of the $L^2$-inner product it is enough to show that $$\langle \psi, p(\Lambda_{\omega_B}\mu^*F_{\scH})\rangle = \langle \psi, p(\Delta_{\scV}\Lambda_{\omega_B}\mu^*F_{\scH})\rangle,$$ where $\psi \in C^{\infty}_E(X,\R)$ and by definition the inner product is given by $$\langle \psi, \phi \rangle = \int_X \psi\phi \omega_X^m \wedge \omega_B^n = \int_B \left( \int_{X/B} \psi\phi \omega_X^m\right) \omega_B^n.$$ But since $\psi \in C^{\infty}_E(X, \R)$ is a fibrewise holomorphy potential, we have \begin{align*} \langle \psi, p(\Delta_{\scV}\Lambda_{\omega_B}\mu^*F_{\scH})\rangle &= \langle \psi, \Delta_{\scV}\Lambda_{\omega_B}\mu^*F_{\scH}\rangle, \\ &= \langle \Delta_{\scV}\psi,\Lambda_{\omega_B}\mu^*F_{\scH})\rangle, \\ &=  \langle \psi,\Lambda_{\omega_B}\mu^*F_{\scH})\rangle, \\ &= \langle\psi ,p(\Lambda_{\omega_B}\mu^*F_{\scH})\rangle, \end{align*} where we have used in turn that $p$ is an orthogonal projection, the vertical Laplacian is self-adjoint with respect to $\langle \cdot, \cdot \rangle$ as it is self-adjoint on each fibre, and elements of $C^{\infty}_E(X)$ are eigenfunctions of $\Delta_{\scV}$ since each fibre is Fano K\"ahler-Einstein \cite[Remark 6.13]{tian-book}. \end{proof}

The manner in which the optimal symplectic connection condition first arose was through the following:

\begin{proposition}\cite{morefibrations}\label{prop:expansion-scalar} The scalar curvature of $\omega_k = k\omega_B + \omega_X$ admits a $C^j$-expansion in powers of $k$ for any $j$$$S(k\omega_B + \omega_X) = S(\omega_b) + k^{-1}(\psi_B + \psi_E + \psi_R) + O(k^{-2}),$$ with $\psi_B \in C^{\infty}(B), \psi_E\in C^{\infty}_E(X), \psi_R \in C^{\infty}_R(X)$. Moreover, we have $$\psi_E = p(\Lambda_{\omega_B} \Delta_{\scV} (\mu^*F_{\scH}) + \Lambda_{\omega_B} \rho_{\scH}).$$
\end{proposition}

Invariantly, as the kernel of $\scD_b$ consists of (real) holomorphy potentials, one can consider $E \to B$ as the vector bundle with fibre consisting of real holomorphic vector fields (with real holomorphy potential); a relatively K\"ahler metric then allows one to consider this as the vector bundle whose fibres are the associated holomorphy potentials. For another relatively K\"ahler metric $\omega_X + \ddb \phi$, for clarity we will thus denote $E_{\phi}$ the associated vector bundle whose fibres consist of holomorphy potentials with respect to $\omega_b + \ddb \phi_b$. Of course, there is a natural identification between $E_{\phi}$ and $E$.

We now consider the linearisation of the optimal symplectic connection operator at $\omega_{\phi}= \omega+ \ddb \phi$ $$\phi \to p_{\phi}(\Lambda_{\omega_B} \Delta_{\scV_{\phi}} (\Lambda_{\omega_B}\mu^*F_{\scH_{\phi}}) + \Lambda_{\omega_B} \rho_{\scH_{\phi}}) \in C^{\infty}_{E_{\phi}}(X,\R),$$ using the obvious notation. Denote by $\nabla_{\scV}$ the vertical gradient operator  $$\nabla_{\scV}: C^{\infty}(X,\R) \to \Gamma(\scV),$$ where $\Gamma(\scV)$ denotes smooth sections of the vertical tangent bundle $\scV$, defined by fibrewise by $$\nabla_{\scV} (\phi)|_{X_b} =\nabla_{\omega_b} (\phi|_{X_b} ) ,$$ and then gluing. The vector bundle $\scV^{1,0} \subset TX^{1,0}$ is a holomorphic subbundle, so one can define $$\scR = \bar\partial \nabla_{\scV}^{1,0}.$$ Note that for $\phi \in C^{\infty}_E(X)$, if $\scR ( \phi ) =0$, since $\nabla^{1,0}_{\scV}$ is already holomorphic on each fibre, $\phi$ is the potential for a global holomorphic vector field on $X$. Similarly by holomorphicity on each fibre, there component of $\bar\partial \nabla_{\scV}^{1,0}\phi$ lying in the vertical (holomorphic) tangent bundle automatically vanishes, so asking that $\scR(\phi) = 0$ asks for holomorphicity in the remaining horizontal directions.

The metrics $(\omega_X)|_{\scV}$ and $\omega_B$ induce a metric on $\scV^{1,0} \otimes \pi^* \Lambda^{0,1} B;$ note $\scR (\phi) \in \Gamma(\scV^{1,0} \otimes \pi^* \Lambda^{0,1} B)$ for $\phi \in C^{\infty}_E(X)$. 

Denote by $$\L_k = \L_0 + k^{-1}\L_1+\hdots$$ the linearisation of the scalar curvature of $\omega_k$. 

\begin{proposition}\label{rstar-right-invertibility}\cite[Theorem 4.9]{morefibrations} Let $\omega_X$ be an optimal symplectic connection. Then for $\phi, \psi \in C^{\infty}_E(X)$, the operator $$p \circ \L_1:C^{\infty}_E(X) \to C^{\infty}_E(X)$$ satisfies $$\int_X \phi (p \circ \L_1)(\psi) \omega_X^m \wedge \omega_B^n = \int_X \langle \scR (\phi), \scR(\psi)\rangle \omega_X^m \wedge \omega_B^n,$$ where the $ \langle \scR (\phi), \scR(\psi)\rangle$ is taken using the natural metric on $\scV^{1,0} \otimes \pi^* \Lambda^{0,1} B$ described above.

Moreover, the operator  $p \circ \L_1$ is a second order self-adjoint elliptic operator on the bundle $E$, whose kernel consists of fibrewise holomorphy potentials which correspond to global holomorphy potentials on $X$ itself. 
\end{proposition}

\noindent In particular, the operator $p\circ \L_1$ is real, and its kernel agrees with that of the operator $\scR$.

\subsection{Relatively cscK metrics}
In this short section we discuss the leading order term $\scL_0$ of the expansion of the linearisation of the scalar curvature of $\omega_k$. 

Similarly to the vertical Laplacian operator $\Delta_{\scV}$ defined above, we can define a vertical Lichnerowicz operator $\scD_{\V}^* \scD_{\scV}$ in such a way that $$\scD_{\V}^* \scD_{\scV} \phi|_{X_b} = \scD_{\omega_b}^* \scD_{\omega_b} (\phi|_{X_b}).$$ These glue to form a smooth operator, as the metrics are varying smoothly, with kernel of constant dimension.

In the decomposition $$C^{\infty}(X, \R) = C^{\infty}(B)\oplus C_E^{\infty}(X)\oplus C_R^{\infty}(X),$$ the two first components consist of the kernel of $\scD_{\V}^* \scD_{\scV} $. On the component $C^{\infty}_R (X)$, on the other hand, for any $\phi$  the restriction $\phi|_{X_b} $ is in the image of $\scD_{\omega_b}^* \scD_{\omega_b}.$ In fact, there is a unique $\psi_b$ orthogonal to the kernel of $\scD_{\omega_b}^* \scD_{\omega_b}$ such that $\scD_{\omega_b}^* \scD_{\omega_b}  ( \psi_b) =\phi|_{X_b}.$ The functions $\psi_b$ glue together to a smooth function $\psi \in C^{\infty}_R (X)$ such that $\scD_{\V}^* \scD_{\scV}  (\psi) = \phi,$ again using that the kernel of the operator is of constant dimension. The function $\psi$ is the unique function with this property, as it is uniquely determined on each fibre by the condition that the restriction is orthogonal to $\scD_{\omega_b}^* \scD_{\omega_b}.$ In conclusion, this proves:

\begin{proposition}\label{R-invertibility} If $\omega_X$ is relatively cscK, the operator $\scD^*_{\V}\scD_{\scV}$ is invertible on $C^{\infty}_R(X,\R)$. \end{proposition}

\subsection{Twisted extremal metrics}\label{sec:twisted-extremal}

We now turn to the base manifold $(B,\beta)$. We do not impose any conditions whatsoever on the K\"ahler metric $\omega_B$, but nevertheless it will be important that one can view $\omega_B$ as a ``canonical metric'' in its own right. The manner in which we do this, following \cite{stablefibrations, hashimoto} is via twisted cscK metrics. Twisted extremal metrics will also play an important role in the present work, and so we discuss these metrics at that level of generality. 

\begin{definition} Let $\zeta$ be a closed, semi-positive form on $B$. We say that $\omega_B \in \beta$ is a
\begin{enumerate}
\item \emph{twisted cscK metric} if $S(\omega_B) - \Lambda_{\omega_B}\zeta$ is constant;
\item \emph{twisted extremal metric} if $$\bar \partial \nabla^{1,0}(S(\omega_B) - \Lambda_{\omega_B}\zeta) = 0.$$
\end{enumerate}
If $\omega_B$ is a twisted extremal metric, the associated holomorphic vector field $$\nabla^{1,0}(S(\omega_B) - \Lambda_{\omega_B}\zeta)$$ is called the \emph{twisted extremal vector field}.
\end{definition}

Twisted extremal vector fields are best viewed geometrically as arising from morphisms between manifolds. Let $$q: (B,\beta) \to (\scM, \gamma_{\M})$$ be a morphism, with $\gamma_{\scM}$ a K\"ahler class on $\scM$. Then if $\zeta_{\scM} \in \gamma_{\M}$, its pullback $\zeta = q^*\zeta_{\scM}$ is a closed semi-positive form on $B$.

\begin{definition} We define the automorphism group of $q$ to be $$\Aut(q) = \{ g \in \Aut B:  q \circ g = q\}.$$ The connected component of the identity in $\Aut(q)$ is denoted $\Aut_0(q)$. The Lie algebra $\mfh_q \subset \mfh$ is defined to consist of holomorphic vector fields whose flow lies in $\Aut_0(q)$.
 \end{definition}
 
 \noindent We will assume throughout that the twisted extremal vector field lies in $\mfh_q$; it seems this is the only case of geometric interest, and will always be satisfied in our constructions.

The automorphism group of the map then gives a geometric understanding of the geometry of twisted extremal metrics. Denote by $$\Isom_0(q, \omega_B) = \Isom(\omega_B) \cap \Aut_0(q)$$ the isometry group of the map with respect to $\omega_B$. 

\begin{theorem}\label{thm:automorphisms-of-maps} Suppose $\omega_B$ is a twisted extremal metric. Then 
\begin{enumerate}
\item $\Isom_0(q, \omega_B) $ is a maximal compact subgroup of $\Aut_0(q)$;
\item if $\omega_B, \omega_B' \in \beta$ are twisted extremal metrics with the same twisted extremal vector field, then there is a $g \in \Aut_0(q)$ with $g^*\omega_B = \omega_B'.$
\end{enumerate}
\end{theorem}

\noindent The first statement is due to the authors \cite[Corollary 4.2]{fibrations}. The uniqueness result is originally due to Keller in the case that either $\Aut(B,\beta)$ is discrete or $\eta$ is positive, with $\beta = c_1(L)$ the first Chern class of an ample line bundle \cite{keller}. In general the uniqueness statement follows from the work of Berman-Berndtsson \cite{berman-berndtsson}, and the geometric version of the uniqueness statement can be found as \cite[Corollary 3.8]{fibrations}. We will require, and hence will prove, more precise uniqueness statements in Section \ref{sec:uniqueness}. One can also find results concerning reductivity of the relevant Lie algebras of holomorphic vector fields in \cite[Theorem 4.1]{fibrations} and \cite[Proposition 7]{datar-szekelyhidi}; these will not be needed in the present work. 

We return now to the setting of Section \ref{sec:optimal}, so that $\pi: (X,\alpha) \to (B,\beta)$ is a holomorphic submersion and $\omega_X \in \alpha$ is a relatively cscK metric. 

\begin{theorem}\cite[Proposition 4.3]{morefibrations}\cite{fine2, fibrations}\label{weil-petersson-form} Let $$\zeta = \int_{X/B} \rho_{\scH} \wedge \omega_X^m.$$ Then $\zeta$ is a closed, semi-positive $(1,1)$-form on $B$, which is independent of choice of relatively cscK metric $\omega_X$ \end{theorem}

The manner in which this is proved is to show that $\zeta$ is the pullback to $B$ of the Weil-Petersson metric via the moduli map $q: B \to \M$, with $\M$ the moduli space of polarised manifolds admitting a constant scalar curvature K\"ahler metric constructed in \cite{fujiki-schumacher, moduli}. 

Thus in some sense for holomorphic submersions the most natural requirement for $\omega_B$ would be to ask that $S(\omega_B) - \Lambda_{\omega_B} \zeta$ is a constant function, with $\zeta$ as above given through the fibre integral. This is precisely the condition used in \cite{fine1, fibrations}. However, the purpose of this paper is to study optimal symplectic connections in general, without any hypotheses on $\omega_B$. Thus we rely on the following crucial result, which allows us to view an arbitrary $\omega_B$ as a twisted cscK metric with a \emph{different} twist.

\begin{proposition}\label{hashimoto}For all $j \gg 0$, there is a K\"ahler metric $\xi \in j\beta$ such that $$ S(\omega_B) - \Lambda_{\omega_B} \zeta = \Lambda_{\omega_B} \xi + c_{\Omega},$$ with $c_{\Omega}$ the appropriate topological constant.
\end{proposition}

\noindent This is due to Hashimoto when $\zeta = 0$ \cite[Proposition 1]{hashimoto}; the proof in the general case is identical \cite[Proposition 2.5]{stablefibrations}.

We end our discussion with the linearisation of the operator. Let $\upsilon$ be a non-negative closed $(1,1)$-form on a compact K\"ahler manifold $Y$. For our applications, we will consider the two cases when $Y=B$ and $\upsilon = \zeta + \xi$, and when $Y=X$ and $\upsilon = \pi^* \xi.$ For a K\"ahler metric $\omega$ on $Y$, define an operator $$\scL_{\upsilon} = \scL_{\upsilon, \omega} : C^{k,\alpha}(Y,\C) \to C^{k-4,\alpha}(Y,\C)$$ by 
$$\scL_{\upsilon} = - \scD^*\scD + \frac{1}{2} \langle \nabla \Lambda_{\omega} \upsilon, \nabla \phi\rangle + \langle i\ddbar \phi, \upsilon \rangle. $$ 
\begin{theorem}\label{linearisation-base} Suppose $\omega$ is a twisted cscK metric, with twist $\upsilon$. The operator $\scL_{\upsilon}$  with $k \geq 4$, is a real fourth-order self-adjoint elliptic operator, with kernel given by the holomorphy potentials $h$ for $\omega$ such that $\| h \|_{\upsilon} =0$, where $\| \cdot \|_{\upsilon}$ is the semi-norm defined by $\upsilon$. Moreover, $\L_{\upsilon}: C^{\infty}(B,\R) \to C^{\infty}(B,\R) $ is the linearisation of the twisted scalar curvature at the twisted cscK metric $\omega$.
\end{theorem}
\noindent This result was originally proved by Keller \cite{keller}, and later reproved in slightly more generality by Hashimoto \cite{hashimoto} and  the authors \cite[Proposition 4.3]{fibrations}. 

\begin{corollary}\cite[Proposition 3.5]{fibrations}\label{kernel-geometrically} Suppose $q: Y \to W$ is a morphism, and $\upsilon$ is the pullback of a K\"ahler metric from $W$. Then $\ker \scL_{\upsilon}$ corresponds to holomorphic vector fields $\mfh_q \subset \mfh$ whose flow lies in $\Aut_0(q)$.\end{corollary}

In the case when $Y=B$, the twist $\upsilon$ will in our situation of interest be K\"ahler (which can be interpreted geometrically via the identity map $(Y,\omega) \to (Y,\upsilon)$, which for example has no non-trivial automorphisms). The kernel is then simply the constants.

When $\upsilon$ is not positive (but still non-negative), the kernel of $\L_{\upsilon}$ will not necessarily consist of just the constants. To rectify this, letting $\overline{\mathfrak{h}}$ be the kernel of $\L_{\upsilon}$ at $\omega$, we will consider the operator 
$$\Psi : C^{k, \alpha}(Y) \times \overline{\mathfrak{h}} \to C^{k-4,\alpha}(Y),$$
or its analogue between Sobolev spaces, given by
$$ (\phi, h) \mapsto S \left( \omega + \ddb \phi \right) - \Lambda_{\omega + \ddb \phi} \left( \upsilon \right) - \frac{1}{2} \langle \nabla h, \nabla \phi \rangle - h.$$
Note that a zero of $\Psi$ is precisely a twisted extremal metric, by Lemma \ref{lem:changeinpot}. The linearisation at $(0, f)$ is 
\begin{align}\label{eq:linearisation}
 (\phi, h) \mapsto \scL_{\upsilon} (\phi) - \frac{1}{2} \langle \nabla \left( S(\omega) - \Lambda_{\omega} (\upsilon ) - f \right), \nabla \phi \rangle - h .
\end{align}
In particular, at a twisted extremal metric, with $f$ being the potential for the extremal vector field, the linearisation is $$ (\phi, h) \mapsto \scL_{\upsilon} (\phi) - h. $$
%
%
%
%\begin{remark} Given a submersion $\pi: X \to B$ and a relatively cscK metric $\omega_X \in \alpha$, there are two natural morphisms of interest. The first, of course, is $\pi: X \to B$ itself. The other is the moduli morphism $q: B \to \scM$. Starting from a twisted cscK metric on $B$ (with twist $\zeta + \xi$), we will construct twisted extremal metrics on $X$ (with twist $\pi^*\xi$). Thus the geometry of both morphisms will be important for us, and we will need to understand the geometry of the twisted extremal operator in both situations. In particular, we will be more interested in \emph{uniqueness} of twisted extremal metrics on the morphism $\pi: X \to B$ than on $q: B \to \scM$, and we will be more interested in the \emph{linearised} operator $\L_{\zeta+\xi}$ on $B$.
%\end{remark}

\section{Twisted extremal metrics on submersions}\label{sec:adiabatic}
Our setup throughout this Section is as follows:

\begin{enumerate}
\item $X \to B$ is a holomorphic submersion between compact complex manifolds;
\item $\alpha$ is a relatively K\"ahler class on $X$ and $\beta$ is a K\"ahler class on $B$;
\item $\omega_X \in \alpha$ is a relatively K\"ahler metric which is cscK on $X_b$ for all $b \in B$;
\item $\omega_B \in \beta$ is a K\"ahler metric on $B$;
\item $\omega_X$ is an optimal symplectic connection;
\item $\xi$ is a K\"ahler metric on $B$ such that $\xi + \zeta$ also is K\"ahler, and $\omega_B$ solves the equation $$S(\omega_B) - \Lambda_{\omega_B} (\zeta + \xi) = c_{\xi}$$ with $c_{\xi} \in \R$ and $\zeta$ the Weil-Petersson metric.
\end{enumerate}

\noindent For any K\"ahler metric $\omega_B$ on $B$,  Proposition \ref{hashimoto} produces a $\xi$ such that the twisted cscK equation of $(vi)$ holds. So this is notation, rather than a true hypothesis.

In this situation, our aim is to construct twisted extremal metrics \emph{on} $X$ \emph{itself}. We will produce such metrics in the K\"ahler class $k \beta + \alpha$ with $k \gg 0$, with twist $\pi^*\xi$. This was achieved in \cite{morefibrations} in the case $\xi = 0$, and many of the techniques are similar. There are two steps: one is to construct an approximate solution to the twisted extremal equation, and the second is to perturb the approximate solution to a genuine solution. For the first step, the main point is that one can understand the twisted scalar curvature on $X$ itself, when the twist is a pullback from the base, through the twisted scalar curvature of the base and the geometry of the fibres. For the second step, the key point is to understand the mapping properties of a right inverse of the linearised twisted extremal operator on $X$; this requires developing new techniques, in comparison with previous work on related questions.

\subsection{The approximate solution}

We begin by constructing approximately twisted extremal metrics on $X$.

\begin{proposition}\label{induction-proposition} For each $r \in \N$, there are functions $$f_1, \hdots, f_r \in C^{\infty}(B, \R), \qquad d_1, \hdots d_r \in C^{\infty}_E(X,\R), \qquad l_1, \hdots, l_r \in C^{\infty}_R(X,\R),$$ fibre holomorphy potentials $$h_1, \hdots, h_r \in C^{\infty}_E(X)$$ and a constant $c \in \R$ such that denoting $$ \phi_{k,r} = \sum_{j=1}^r f_j k^{2-j}, \quad \delta_{k,r} = \sum_{j=1}^r d_j k^{1-j}, \quad \lambda_{k,r} = \sum_{j=1}^r l_j k^{-j}, \quad \eta_{k,r} = c + \sum_{j=1}^r h_jk^{-j},$$ the K\"ahler metric $$\omega_{k,r} = k\omega_B + \omega_X + \ddb (\phi_{k,r} + \delta_{k,r} + \lambda_{k,r})$$ satisfies \begin{equation}\label{inductive-equation}S(\omega_{k,r}) - \Lambda_{\omega_{k,r}}\pi^*\xi = \eta_{k,r} + \frac{1}{2}\langle \nabla \eta_{k,r}, \nabla (\phi_{k,r} + \delta_{k,r} + \lambda_{k,r})\rangle_{\omega_k} + O(k^{-r-1}).\end{equation}
\end{proposition}

\noindent Here the expansion is meant pointwise; however, in \cite[Section 5]{fine1}, Fine shows that these estimates improve to global $C^l$-estimates, which is important when perturbing to genuine solutions in Section \ref{sec:perturb}. 

What we mean for each $h_j$ to be a \emph{fibre holomorphy potential} is that $h_j$ is a holomorphy potential with respect to $k\omega_B + \omega_X$ (with $k$ taken such that this form is K\"ahler) for a real holomorphic vector field $v_j$ on $X$ whose flow $\rho(t)$ satisfies $\rho(t) \circ \pi = \pi$; that $h_j$ is then actually independent of $k$, as the notation suggests is justified by Lemma \ref{fibreholomorphypotential} below. Note that such $v_j$ are simply the real holomorphic vector fields on $X$ which lie in $\scV \subset TX$.

For Proposition \ref{induction-proposition} to construct ``approximate twisted extremal metrics'', by Lemma \ref{lem:changeinpot} one needs $\eta_{k}$ to be a holomorphy potential with respect to $\omega_k$. The following simple Lemma establishes this, and is implicit in \cite{morefibrations}. The Lemma is proven explicitly in the case $X$ and $B$ are projective in \cite[Proposition 3.11]{stablefibrations}; the proof given there applies also to more singular algebro-geometric fibrations.

\begin{lemma}\label{fibreholomorphypotential}
Let $h \in C^{\infty}_E(X,\R)$ be a fibre holomorphy potential. Then $h$ is a holomorphy potential on $X$ with respect to $k\omega_B + \omega_X$ for all $k$ for which the form is K\"ahler. 
\end{lemma}

\begin{proof} The fibre holomorphy potential $h$ corresponds to a global real holomorphic vector field $v$ on $X$. The claim states that the holomorphy potential for $v$ with respect to $k\omega_B + \omega_X$ is actually independent of $k$. If not, then by linearity of the construction the holomorphy potential would be of the form $kh_B + h$. 

Let $\rho(t)$ be the flow of $h$, so that $\rho(t) \circ \pi = \pi.$ Then $$(\rho(t) \circ \pi)^*\omega_B = \pi^*\omega_B.$$

Setting $$\rho(t)^*(k\omega_B + \omega_X) - (k\omega_B + \omega_X) = \ddb \phi(t),$$ the holomorphy potential for $v$ is given by $\rho(t)_*\dot \phi(t)$ (in particular, this quantity is independent of $t$) \cite[Example 4.16]{gabor-book}. Since $$\rho(t)^*(k\omega_B + \omega_X) - (k\omega_B + \omega_X) = \rho(t)^*\omega_X - \omega_X$$ is independent of $k$, it must be the case that $h_B=0$.
\end{proof}

The proof of Proposition \ref{induction-proposition} is inductive. The starting point is the following expansion of the scalar curvature and contraction terms. Recall that $$\omega_k = k\omega_B + \omega_X,$$ $p(\theta)$ is the projection of the curvature quantity used in the definition of an optimal symplectic connection, and $S(\omega_b)$ denotes the function on $X$ whose restriction to any fibre $X_b$ is the scalar curvature of $\omega_b = \omega_X|_{X_b}$.

\begin{lemma}\label{lem:initial-expansion} We have $$S(\omega_k) - \Lambda_{\omega_k} \pi^*\xi = S(\omega_b) + k^{-1}\left(S(\omega_B) - \Lambda_{\omega_B} (\zeta + \xi) + p(\theta) + \psi_{R,1} \right) + O(k^{-2}),$$ for some $\psi_{R,1} \in C^{\infty}_R(X).$ \end{lemma} 

\begin{proof}This follows from the expansions established in \cite{morefibrations}. Indeed, from \cite[Proposition 4.8]{morefibrations} one has $$S(\omega_k) = S(\omega_b) + k^{-1}\left( S(\omega_B) - \Lambda_{\omega_B} \zeta + p(\theta) + \psi_{R,1} \right) + O(k^{-2}),$$ and from \cite[Lemma 4.2]{morefibrations} one has $$\Lambda_{\omega_k} \pi^*\xi = k^{-1} \Lambda_{\omega_B} \xi + O(k^{-2}).$$ Here it is important that $\pi^*\xi$ is pulled back from $B$. 
\end{proof}

\noindent By our assumptions on $\xi$ and $\omega_X$, both $S(\omega_B) - \Lambda_{\omega_B} (\zeta + \xi)$ and $p(\theta)$ are constant. Thus the non-constant $k^{-1}$-term is $\psi_{R,1}$. In order to obtain a twisted cscK metric to order $k^{-1}$, we add a potential $\phi_{R,1}\in C^{\infty}_R(X,\R)$ to $\omega_k$. This affects the contraction in the following quite trivial way.

\begin{lemma}\label{contraction-behaviour-step-1}  Let $\phi \in C^{\infty}(X,\R)$. Then $$\Lambda_{k\omega_B + \omega_X + k^{-1}\ddb \phi} \xi = k^{-1}\Lambda_{\omega_B}\xi + O(k^{-2}).$$ \end{lemma}

\begin{proof} The proof is identical to \cite[Lemma 4.2]{morefibrations}, where the case $\phi_R=0$ was considered. We briefly give the details.

One  first writes $$\Lambda_{k\omega_B + \omega_X+ k^{-1}\phi_{R,1}} \xi = (m+n)\frac{\xi \wedge (k\omega_B + \omega_X+ k^{-1}\phi_{R,1})^{m+n-1}}{(k\omega_B + \omega_X+ k^{-1}\phi_{R,1})^{m+n}},$$ then uses that $\beta$ is purely horizontal to calculate $$(m+n)\frac{\xi \wedge (k\omega_B + \omega_X+ k^{-1}\phi_{R,1})^{m+n-1}}{(k\omega_B + \omega_X+ k^{-1}\phi_{R,1})^{m+n}} = k^{-1} \frac{\xi \wedge (\omega_X)_{\V}^m\wedge \omega_B^{n-1}}{(\omega_X)_{\V}^m\wedge \omega_B^{n} } + O(k^{-2}).$$ \end{proof}

\begin{corollary}\label{cor:R1-kill}There is a function $l_1 \in C^{\infty}_R(X,\R)$ such that $$S(k\omega_B + \omega_X + k^{-1}\ddb l_1) - \Lambda_{k\omega_B + \omega_X+ k^{-1}l_1}\pi^* \xi  = c + k^{-1}h_1 + O(k^{-2}),$$ with $c_0$ and $h_1$ constants. \end{corollary}

\begin{proof} The follows from  \cite[Proposition 4.8]{morefibrations} combined with the above. Indeed,  since $\omega_X$ is cscK on each fibre, \cite[Proposition 4.8]{morefibrations} produces a function $l_1 \in C^{\infty}_R(X,\R)$ such that $$S(k\omega_B + \omega_X + k^{-1}\ddb l_1) -S(k\omega_B + \omega_X)  = k^{-1}(-\psi_{R,1}) + O(k^{-2}).$$ Lemma \ref{contraction-behaviour-step-1} provides $$\Lambda_{k\omega_B + \omega_X + k^{-1}\ddb \phi} \pi^*\xi = \Lambda_{k\omega_B + \omega_X} \pi^*\xi + O(k^{-2}),$$ giving the result. \end{proof}

Thus we have an approximately twisted cscK metric to order $k^{-1}$, with $f_1 = d_1 = 0$. Writing the $k^{-2}$ term in the expansion of \begin{equation}\label{operator-to-linearise}S(k\omega_B + \omega_X + k^{-1}\ddb l_1) - \Lambda_{k\omega_B + \omega_X+ k^{-1} l_1}\pi^* \xi$$ as $$k^{-2}(\psi_{B,2} + \psi_{E,2} + \psi_{R,2}),\end{equation} we wish to add potentials $f_2, d_2, l_2$ such that the $k^{-2}$ is the twisted extremal equation to order $k^{-2}$. Since this process happens order by order, what is important is to understand the \emph{linearisation} of the operator \begin{equation}\phi \to S(\omega_{k,1}+\ddb \phi) - \Lambda_{\omega_{k,1}+\ddb \phi}\pi^* \xi. \end{equation}

The behaviour of the linearisation changes depending on what function space $\phi$ lies in. If $\phi \in C^{\infty}(B,\R)$, then  Proposition \ref{linearisation-inductive-step} below shows that the operator acts as the linearisation of the base component of the $k^{-1}$-term, i.e. as the linearisation of $$\phi \to S(\omega_B + \ddb \phi) - \Lambda_{\omega_B+\ddb \phi}(\zeta + \xi).$$ By Theorem \ref{linearisation-base}, this is an invertible operator modulo constants, as $\omega_B$ solves the twisted cscK equation with the positive twisting form $\zeta + \xi$. This term is the only novelty in comparison with the previous work \cite{morefibrations}, where the case $\xi=0$ was considered and instead it was assumed that $\omega_B$ itself was twisted cscK with twist $\zeta$. Since in both cases one has a twisted cscK metric, an identical strategy succeeds. In fact, our situation is somewhat simpler, as $\zeta + \xi$ is positive, and therefore the kernel of the twisted Lichnerowicz operator only consists of the constants, without imposing any conditions on the automorphism group of $B$. 

When $\phi\in C^{\infty}_R(X,E)$ is instead a fibrewise holomorphy potential, then the crucial part of the linearisation is the  linearisation of the operator $p(\theta)$ of Lemma \ref{lem:initial-expansion}. As the submersion $X \to B$ may have automorphisms, this operator is \emph{not}, in general, invertible. Instead by Proposition \ref{rstar-right-invertibility}, one can, up to a function in $C^{\infty}_R (X)$, solve $$p\circ \L_1 (d_2) = \psi_{E,2} - h_2,$$ with $h_2$ a fibre holomorphy potential for a global holomorphic vector field on $X \to B$.

This allow us to manage the $C^{\infty} (B)$ and $C^{\infty}_E(X)$ components. The remaining component is now the $C^{\infty}_R(X)$-component, which is dealt with by adding a function $l_1 \in C^{\infty}_R(X)$; just as in Proposition \ref{cor:R1-kill}, when adding such a function, the operator of Equation \eqref{operator-to-linearise} acts as $-\scD^*_{\V}\scD_{\V},$ which is an isomorphism on $C^{\infty}_R(X)$. This allows us to correct the $C^{\infty}_R (X)$-component. 

More precisely, what we need is the following.

\begin{proposition}\label{linearisation-inductive-step} Let $\omega_{k,r}$ be the metric solving Equation \eqref{inductive-equation}, and denote by $\L_{k,r}$ the linearisation of the operator $$\phi \to S(\omega_{k,r}+ \ddb \phi) - \Lambda_{\omega_{k,r}+\ddb \phi} \pi^*\xi.$$ Then $\L_{k,r}$ satisfies the following properties:

\begin{enumerate}
\item there is an expansion $$\L_{k,r} = -\scD^*_{\V}\scD_{\V}(\phi) + k^{-1} D_1(\phi) + k^{-2}D_2(\phi) + O(k^{-3});$$% with $D_1$ depending on $l_1$ and $D_2$ depending on $l_1, l_2$.
\item for $\phi \in C^{\infty}(B,\R)$, we have $D_1(\phi) = 0$ and $$\int_{X/B} \phi \omega_X^m = -\L_{\eta + \xi} (\phi);$$ 
\item for $\phi \in C^{\infty}_E(X,\R)$ (so $\scD^*_{\V}\scD_{\V}(\phi) = 0)$, we have $$p \circ D_1(\phi) = - p \circ \L_1(\phi).$$
\end{enumerate}

\end{proposition}

\begin{proof}

This is proved identically to \cite[Proposition 5.6]{fibrations} and \cite[Proposition 4.11]{morefibrations}, with parts (i) and (ii) building heavily on the lower dimensional work of Fine \cite[Section 3.3]{fine1}). The only difference with \cite[Proposition 4.11]{morefibrations} is the behaviour on functions $\phi \in C^{\infty}(B,\R)$. In the situation considered there, the operator being linearised is $$\phi \to S(\omega_{k,r} + \ddb \phi) - \eta_{k,r} - \frac{1}{2} \langle \nabla \eta_{k,r}, \nabla(\phi_{k,r} + \delta_{k,r} + \lambda_{k,r})\rangle.$$ The assumption of \cite{morefibrations} is then that $S(\omega_B) - \Lambda_{\omega_B} \eta$ is constant, and so the $C^{\infty}(B,\R)$-component of the $k^{-1}$-term is $$S(\omega_B) - \Lambda_{\omega_B} \eta.$$

In our situation, we are linearising $$\phi \to S(\omega_{k,r}+ \ddb \phi) - \Lambda_{\omega_{k,r}+\ddb \phi} \pi^*\xi,$$ and the $C^{\infty}(B,\R)$-component of the $k^{-1}$-term is $$S(\omega_B) - \Lambda_{\omega_B} (\zeta + \xi).$$ The form $\xi$ is chosen such that $S(\omega_B) - \Lambda_{\omega_B} (\zeta + \xi)$ is constant, hence we are \emph{still} working with a solution of the twisted cscK equation on the base. The key ideas used in \cite[Proposition 4.11]{morefibrations}, therefore, apply in our situation: we have a relatively cscK metric, an optimal symplectic connection, and a twisted cscK metric on the base. Thus while superficially the situations appear different, in practice the details are exactly the same. Thus the proof of \cite[Proposition 4.11]{morefibrations} goes through verbatim. 

We note that since $\omega_X$ is an optimal symplectic connection (rather than extremal symplectic connection) and $\omega_B$ is twisted cscK, in the setup of \cite[Section 4.6]{morefibrations} the functions $b_1$ and $h_1$ vanish, removing some of the technicalities.
 \end{proof}

We can now inductively prove Proposition \ref{induction-proposition}.

\begin{proof}[Proof of Proposition \ref{induction-proposition}]

We have already proved the initial step $r=1$, and thus we proceed by induction. Write \begin{align*}S(\omega_{k,r}) - \Lambda_{\omega_{k,r}} \pi^*\xi- \eta_{k,r} &- \frac{1}{2}\langle \nabla \eta_{k,r}, \nabla(\phi_{k,r} + \delta_{k,r} + \lambda_{k,r})\rangle \\ & = k^{-r-1}(\psi_{B,r+1} + \psi_{E, r+1} + \psi_{R,r+1}) + O(k^{-r-2}).\end{align*} We need to choose $f_{r+1}, d_{r+1}$ and $l_{r+1}$ in order to make the $k^{r+1}$-coefficient constant.

We begin with the $C^{\infty}(B,\R)$-term. Since $S(\omega_B) - \Lambda_{\omega_B} (\zeta + \xi)$ is constant, by Theorem \ref{linearisation-base}, the operator $\L_{\zeta + \xi}: C^{\infty}(B,\R) \to C^{\infty}(B,\R)$ is invertible modulo constants. Thus we can find a function $f_{r+1} \in C^{\infty}(B,\R)$ such that $$\L_{\zeta + \xi}(f_{r+1}) = \psi_{B,r+1} + c_{r+1},$$ with $c_{r+1}$ constant. From Proposition \ref{linearisation-inductive-step} $(ii)$, we have \begin{align*}S(\omega_{k,r}&+k^{-r+1}\ddb f_{r+1}) - \Lambda_{\omega_{k,r}+k^{-r+1}\ddb f_{r+1}} \pi^*\xi- \eta_{k,r} - \frac{1}{2}\langle \nabla \eta_{k,r}, \nabla(\phi_{k,r} + \delta_{k,r} + \lambda_{k,r})\rangle \\ & = k^{-r-1}(\psi'_{E, r+1} + \psi'_{R,r+1} + c_{r+1}) + O(k^{-r-2}),\end{align*} with $\psi'_{E, r+1} \in C^{\infty}_E(X,\R)$ and $\psi'_{R,r+1}\in C^{\infty}_R(X,\R)$.

We next turn to the $C^{\infty}_E(X,\R)$ term. The operator of interest to us, $p \circ \L_1$, is not invertible when $X \to B$ admits continuous automorphisms. Nevertheless, Proposition \ref{rstar-right-invertibility} produces a function $d_{r+1}$ and a fibre holomorphy potential $h_{r+1}$ such that $$p \circ \scL_1(d_{r+1}) + h_{r+1} =   \psi'_{E, r+1} .$$ Then by Proposition \ref{linearisation-inductive-step} $(iii)$, the K\"ahler metric  $$\omega_{k,r+1}' = \omega_{k,r}+k^{-r+1}\ddb f_{r+1} + k^{-r} \ddb d_{r+1}$$ satisfies \begin{align*}S(\omega_{k,r+1}') - \Lambda_{\omega_{k,r+1}'} \pi^*\xi-& \eta_{k,r} - \frac{1}{2}\langle \nabla \eta_{k,r}, \nabla(\phi_{k,r} + \delta_{k,r} + \lambda_{k,r})\rangle \\ & = k^{-r-1}(h_{r+1} + \psi''_{R,r+1} + c_{r+1}) + O(k^{-r-2}),\end{align*} with $\psi''_{R,r+1} \in C^{\infty}_R(X,\R).$ In order for this to be as we desire, we need that 
$$k^{-r-1}\langle \nabla h_{r+1}, \nabla(\phi_{k,r+1} + \delta_{k,r+1} + \lambda_{k,r})\rangle =O(k^{-r-2}),$$ 
and that
$$\langle \nabla \eta_{k,r}, \nabla(k^{-r+1} f_{r+1} + k^{-r} d_{r+1})\rangle =O(k^{-r-2}),$$ 
where by definition $$\phi_{k,r+1} = \phi_{k,r} + f_{r+1}k^{-r+1}, \quad \delta_{k,r+1} = \delta_{k,r} + k^{-r}d_{r+1}.$$ This is established  in \cite[p. 40]{morefibrations}.

Finally we turn to the $C^{\infty}_R(X,\R)$-term $\psi_{R,r+1}''$, which is the most straightforward step. Proposition \ref{R-invertibility} produces a function $l_{r+1}$ such that $$\scD_{\V}^*\scD_{\V}(l_{r+1}) = -\psi_{R,r+1}''.$$ Set $$\omega_{k,r+1} = \omega_{k,r}+k^{-r+1}\ddb f_{r+1} + k^{-r} \ddb d_{r+1} + k^{-r} l_{r+1}.$$ By \cite[p. 41]{morefibrations}, we have $$ \langle \nabla \eta_{k,r+1}, \nabla k^{-r-1}l_{r+1}\rangle = O(k^{-r-2}).$$ Thus, with $$\lambda_{k,r+1} = \lambda_{k,r} + k^{-r-1}l_{r+1},$$ we have that $$S(\omega_{k,r+1}) - \Lambda_{\omega_{k,r+1}} \pi^*\xi- \eta_{k,r+1} - \frac{1}{2}\langle \nabla \eta_{k,r+1}, \nabla(\phi_{k,r+1} + \delta_{k,r+1} + \lambda_{k,r+1}) \rangle $$ is $O(k^{-r-2}).$ This proves the inductive step, and hence the proof.
\end{proof}

\subsection{Perturbation}\label{sec:perturb}

Having constructed approximate twisted extremal metrics on $X$, we are now in a position to perturb them to genuine twisted extremal metrics using a quantitative version of the implicit function theorem. 

\begin{theorem}\label{thm:implicit} Let $F: B_1 \to B_2$ be a differentiable  map of Banach spaces whose derivative at $0 \in B_1$ is surjective, with right inverse $P$. Denote 
\begin{enumerate}
\item $\delta'$ the radius of the closed ball in $B$ around the origin on which $F-DF$ is Lipschitz with Lipschitz constant $(2\|P\|)^{-1},$
\item $\delta = \frac{\delta'}{2\|P\|}.$
\end{enumerate}
Then for all $y \in B_2$ with $\|y - F(0)\| < \delta$, there exists a $x\in B_2$ with $\| x \| < \delta'$ satisfying $F(x)=y$.
\end{theorem}

An important consequence for our main results will be the explicit bound one obtains on the distance between the genuine solution $x$ and distance to the approximation $F(0)$. 

\begin{remark}\label{rem:implicit} In fact, one obtains the following statement. Let $\tau' < \delta',$ and put $\tau = \frac{\tau'}{2\|P\|}.$ Then for all $y \in B_2$ with $\|y - F(0)\| < \tau$, there exists a $x\in B_2$ with $\| x \| < \tau'$ satisfying $F(x)=y$. 
\end{remark}
We give a short proof of Remark \ref{rem:implicit} when the map is between Hilbert spaces, and the linearised operator is Fredholm, which will be the case in our application of Theorem \ref{thm:implicit}. We may then assume $DF$ is invertible. Otherwise we replace $B_1$ with the orthogonal complement in $B_1$ to the kernel of $DF$. Since $DF$ is invertible with inverse $P$, we have $\| DF(x) \| \geq \frac{1}{\| P \|} \| x \|$ for all $x$. Thus if $N$ denotes $F - DF$, which is Lipschitz of constant $\frac{1}{2 \| P \|}$ in the ball of radius $\tau' < \delta'$, we have
\begin{align*} \| F(x) - F(0) \| &= \| N(x) - N(0) + DF(x) \| \\
& \geq \| DF (x) \| - \| N(x) - N(0) \| \\
& \geq \frac{1}{\| P \|} \| x \| - \frac{1}{2 \| P \|} \| x \| \\
&= \frac{1}{2 \| P \|} \| x \|
\end{align*}
and so the ball of radius $\tau'$ hits at least the ball of radius $\tau$, as we already know $F$ is surjective on the balls of larger radii from Theorem \ref{thm:implicit}.

The $F$ that we will consider is the map
$$ \Psi_{k,r} : L^2_{l+4} \left( X, \omega_{k,r} \right) \times \overline{ \mathfrak{h} }_{\pi} \to L^2_l \left( X, \omega_{k,r} \right)$$ 
given by 
$$ (\phi, h) \mapsto S \left( \omega_{k,r} + \ddb \phi \right) - \Lambda_{\omega_{k,r} + \ddb \phi} \left( \pi^* \xi \right) - \frac{1}{2} \langle \nabla h, \nabla \phi \rangle - h.$$
Here $\overline{\mathfrak{h}}_{\pi}$ are the potentials, with respect to $\omega_{k,r}$, for holomorphic vector fields in $\mathfrak{h}_{\pi},$ the holomorphic vector fields whose flow lies in $\Aut_0(\pi)$. The linearisation at $(0, f)$ is 
$$ (\phi, h) \mapsto  \scL_{\xi} - \langle \nabla \left( S(\omega_{k,r}) - \Lambda_{\omega_{k,r}} (\pi^* \xi ) - f \right), \nabla \phi \rangle - h ,$$
where $\scL_{\xi}$ is the $\pi^* \xi$-twisted Lichnerowicz type operator. Note that $\ker \scL_{\xi} = \overline{\mathfrak{h}}_{\pi}$ by Corollary \ref{kernel-geometrically}. At $$f = \eta_{k,r} + \frac{1}{2} \langle \nabla \left( \eta_{k,r} \right), \nabla \left( \phi_{k,r} + \lambda_{k,r} + \delta_{k,r} \right) \rangle $$  the holomorphy potential with respect to $\omega_{k,r}$ constructed in Proposition \ref{induction-proposition}, we see that the linearisation is a perturbation of the operator 
$$F_{k,r}: (\phi, h) \mapsto \scL_{\xi} (\phi) - h.$$ 
Indeed, $\omega_{k,r}$ was constructed so that $ S(\omega_{k,r}) - \Lambda_{\omega_{k,r}} (\pi^* \xi ) - f $ is $O(k^{-r-1})$, and so $F_{k,r}$ can be assumed to approximate the linearisation at any desired order, by restricting to $r$ sufficiently large (in fact $r=3$ will do).

Thus it suffices to prove the required estimates for this operator instead. From the mapping properties of $\scL_{\xi}$, it follows that $F_{k,r}$ is surjective, with right inverse $Q_{k,r}$ that sends a $\psi$ orthogonal to the kernel of $\scL_{\xi}$ to the unique $\phi$ orthogonal to the same kernel with $\scL_{\xi} (\phi) = \psi$. The actual linearisation is then also surjective for large $k$, with right inverse $P_{k,r}$. The claim that $F_{k,r}$ is surjective follows from Theorem \ref{linearisation-base}, which implies that $\scL_{\xi}$ is invertible orthogonal to its kernel, with its kernel given precisely by $\overline{\mathfrak{h}}_{\pi}$.

To apply Theorem \ref{thm:implicit}, we will need a bound on the right inverse $P = P_{k,r}$ of the linearised operator. We will show
\begin{proposition}\label{prop:rightinversebound} For each $r$ sufficiently large, there exists a $C > 0$ such that for all $k \gg 0$
\begin{align*} \| P_{k,r} (\phi) \|_{L^2_{l} (\omega_{k,r}) } \leq C k^3 \| \phi \|_{L^2_{l+4} (\omega_{k,r}) } .
\end{align*}
\end{proposition}
This will be achieved via a lower bound for the first eigenvalue of $- \scL_{\xi}$.

\begin{lemma}\label{lem:PI} For each $r$ sufficiently large, there exists a $C > 0$ such that for all $k \gg 0$
\begin{align}\label{eq:PI} \langle \phi, - \scL_{\xi} (\phi) \rangle_{L^2 (\omega_{k,r}) } &\geq  C k^{-3} \| \phi \|_{L^2 ( \omega_{k,r} )}^2 ,
\end{align}
for all $\phi$ that are $L^2 ( \omega_{k,r} )$-orthogonal to the kernel of $\scL_{\xi}$.
\end{lemma}

\begin{proof} We will use the following integration by parts formula \cite[Equation 3.5]{fibrations}, valid for any K\"ahler form $\omega$ on $X$:
\begin{align*} \langle \psi, \scL_{\xi} (\phi) \rangle_{L^2 (\omega) } =-  \langle \scD \psi, \scD \phi \rangle_{L^2 (\omega)} - \int_X \langle \nabla \psi, \nabla \phi \rangle_{\xi} \omega^n.
\end{align*}
The operators $\scD$ and $\nabla$ are with respect to $\omega$ in the above. In particular, picking $\psi = \phi$ we obtain
\begin{align*} \langle \phi, - \scL_{\xi} (\phi) \rangle_{L^2 (\omega) } &\geq  \| \scD \phi \|^2 .
\end{align*}

We apply this to $\omega = \omega_{k,r}$. By \cite[Lemma 6.7]{fine1}, there is a lower bound
\begin{align*} \| \scD \phi \|_{L^2 (\omega_{k,r} ) }^2 \geq C k^{-3} \| \scD \phi \|_{L^2 ( \omega_{k,r} )}^2 ,
\end{align*}
valid for any $\phi$ that is orthogonal to the kernel of $\scD^* \scD$. Thus, we obtain a lower bound 
\begin{align*} \langle \phi, - \scL_{\xi} (\phi) \rangle_{L^2 (\omega_{k,r} ) } &\geq  C k^{-3} \| \phi \|_{L^2 ( \omega_{k,r} )}^2 ,
\end{align*}
for such $\phi$.

This does not quite prove what we want, as the kernels of $\scD^* \scD$ and $\scL_{\xi}$ may not coincide -- the latter could be strictly contained in the former. Next we assume $\phi$ is in the kernel of $\scD^* \scD$, but still orthogonal to the kernel of $\scL_{\xi}$. Then $\phi$ is a holomorphy potential on $X$. We may assume $r=0$, as the statement for general $r$ is a perturbation of this.

Thus $\phi$ is a holomorphy potential for $\omega_k = \omega_X + k \omega_B$. Writing $\phi = \phi_X + \phi_B$, we claim that the vertical component $\phi_X$ vanishes. This will follow from the fact that $\phi_X$ is itself a holomorphy potential on $X$, and hence it is in the kernel of $\scL_{\xi}$. As $\phi$ is orthogonal to this kernel and the decomposition $\phi = \phi_X + \phi_B$ is orthogonal, it follows that $\phi_X=0$. 

To see that $\phi_X$ is a holomorphy potential, we first note that as $\phi$ and $\phi_X$ restrict to the same function on fibres, $\phi_X$ is a section of $C^{\infty}_E (X)$. Using the asymptotic expansion \ref{linearisation-inductive-step}, we see that $p \circ \L_1 (\phi_X) =0$. But then $\phi_X \in \ker \scL_{\xi},$ which is what we wanted to show.

We have
\begin{align*} \langle \phi, - \scL_{\xi} (\phi) \rangle_{L^2 (\omega_k) } =   \int_X  | \nabla \phi |^2_{\xi} \omega_k^n.
\end{align*}
Now $\pi^* \xi$ is the pullback of a positive form on $B$, and thus $$| \nabla \phi |^2_{\xi} \geq C | \nabla \phi |^2_{\omega_B} = C k^{-1} | \nabla \phi |^2_{k \omega_B}.$$ But we have $| \nabla \phi |^2_{k \omega_B} = | \nabla \phi |^2_{\left( \omega_X \right)_{\scV} + k \omega_B},$ where $\left( \omega_X \right)_{\scV} + k \omega_B$ is the product Riemannian metric on $TX = \scV \oplus \scH$, since $\phi$ is pulled back from $B$. Using the uniform equivalence of this product Riemannian metric and $\omega_k$ \cite[Lemma 6.2]{fine1}, we therefore obtain that for some possibly different constant $C$, we have
\begin{align*} \langle \phi, - \scL_{\xi} (\phi) \rangle_{L^2 (\omega_k) } \geq C k^{-1}  \| \nabla \phi \|^2_{L^2( \omega_k)} 
\end{align*}
for the $\phi$ we are considering. Using the Poincar\'e inequality for the Laplacian \cite[Lemma 6.5]{fine1}, which applies because $\phi$ in particular is orthogonal to the constants, we thus obtain for $\phi$ in the kernel of $\scD^* \scD$, but orthogonal to the kernel of $\scL_{\xi}$, there is a bound of the form
\begin{align*} \langle \phi, - \scL_{\xi} (\phi) \rangle_{L^2 (\omega_k) } \geq C k^{-2}  \|  \phi \|^2_{L^2( \omega_{k})} .
\end{align*}
Thus Equation \eqref{eq:PI} is valid for all $\phi$ orthogonal to the kernel of $\scL_{\xi}$, which is what we wanted to prove.
\end{proof}

Proposition \ref{prop:rightinversebound} will now follow by combining Lemma \ref{lem:PI} with the following Schauder estimate
\begin{align}\label{eq:SE} \| \phi \|_{L_{l+4}^2 (\omega_{k,r}) } \leq C \left(  \| \phi \|_{L^2 (\omega_{k,r}) } + \| \scL_{\xi} (\phi) \|_{L_{l}^2 (\omega_{k,r}) } \right) .
\end{align}
This estimate follows directly from the analogous estimate for $\scD^* \scD$, see \cite[Lemma 6.8]{fine1}, since the twisting form is pulled back from $B$, and hence $- \scL_{\xi} - \scD^* \scD$ is $O(k^{-1})$. 

To establish Proposition \ref{prop:rightinversebound}, we apply \eqref{eq:SE} to $\phi = Q_{k,r} (\psi).$ Note that the existence of $Q_{k,r}$ is already known from the mapping properties of $\scL_{\xi}$. Writing $$ \psi = \psi_1 + \psi_2,$$ where $\psi_2$ is in the kernel of $\scL_{\xi}$ and $\psi_1$ is $L^2( \omega_{k,r})$-orthogonal to this kernel, and similarly for $\phi$, we know $\phi_1 = Q_{k,r} (\psi_1)$ and $\phi_2 = \psi_2$. The Schauder estimate implies that
\begin{align*} \| Q_{k,r} (\psi ) \|_{L_{l+4}^2 (\omega_{k,r}) } \leq C \left(  \| Q_{k,r} (\psi_1) \|_{L^2 (\omega_{k,r}) } + \| \psi_2 \|_{L^2 (\omega_{k,r}) } + \| \psi_1 \|_{L_{l}^2 (\omega_{k,r}) } \right) .
\end{align*}
Lemma \ref{lem:PI} implies that $$\| Q_{k,r} (\psi_1) \|_{L^2 (\omega_{k,r}) } \leq C k^3 \| \psi_1 \|_{L^2 (\omega_{k,r}) },$$ and thus the result follows for $Q_{k,r}$ by the fact that $\| \psi_i \|_{L^2} \leq \| \psi \|_{L^2}$ as the $\psi_i$ are $L^2$-orthogonal. It therefore also holds for $P_{k,r}$ as the actual linearised operator is asymptotic to the negative of the Lichnerowicz type operator $\scL_{\xi}$ we have established the bounds for.

We now have the required bound for $P_{k,r}$. The final thing we need in order to apply Theorem \ref{thm:implicit} is the Lipschitz bound, which is simply a consequence of the mean value theorem.
\begin{lemma}\label{lem:lipschitz} 
Let $N_{k,r} = \Psi_{k,r} - D\Psi_{k,r}$. There exists constants $c,C > 0$ such that for all $k \gg 0$ the following holds. If $\phi_i \in L^2_{l+4} \left( X, \omega_{k,r} \right) \times \overline{\mathfrak{h}}_{\pi}$ satisfy $\| \phi_i \| \leq c,$ then 
\begin{align*} \| N_{k,r} (\phi_1) - N_{k,r} (\phi_2) \| \leq C \left( \| \phi_1 \| + \| \phi_2 \| \right) \| \phi_1 - \phi_2 \|. 
\end{align*}
\end{lemma}

We now prove the main result of this section.
\begin{theorem}\label{thm:solutions} For each $l$ and for each $r$ sufficiently large, there exists $k_0$ such that for all $k \geq k_0$, there exists a twisted extremal metric $\widetilde \omega_k = \omega_{k,r} + \ddb \psi_{k,r}$ such that the solutions satisfy
\begin{align*} \| \psi_{k,r} \|_{L^2_{l+4}} \leq C k^{-3}.
\end{align*}
\end{theorem}
\begin{proof} Let $\delta'$ be as in the statement of Theorem \ref{thm:implicit}. Lemma \ref{lem:lipschitz} implies that there is a $c$ such that for all $\lambda > 0$ sufficiently small, $\Psi_{k,r} - D \Psi_{k,r}$ is Lipschitz of constant $c\lambda$. Thus $\delta' \geq c k^{-3}$, for some potentially different constant $c$, using Proposition \ref{prop:rightinversebound}. Pick $\tau' = ck^{-3} \leq \delta'$, and put $\tau = \frac{\tau'}{2 \| P_{k,r} \|}$. Again by Proposition \ref{prop:rightinversebound}, we have that $\tau \geq c k^{-3} \tau' = C k^{-6} $ for some new constants $c,C$. 

Using Remark \ref{rem:implicit}, we can solve $\Psi_{k,r} (\phi, h) = \Psi_{k,r}(0) + f$ for any $f$ in the ball of radius $\tau$, with $(\phi, h)$ in the ball of radius $\tau' = c k^{-3}$. In particular, then, this applies to $f$ in the ball of radius $C k^{-6} $, for some $C$. Thus to ensure that we can solve the twisted extremal equation, we need that $ \| S(\omega_{k,r}) - \Lambda_{\omega_{k,r}} (\pi^* \xi ) - h_{k,r} \|  < C k^{-6}$, which by Proposition \ref{induction-proposition} holds if $r \geq 7.$ The result follows.
\end{proof}

\noindent Thus we obtain existence of twisted extremal metrics, and also a bound on how the genuine twisted extremal metrics compare with our approximate solutions. This also implies similar bounds in H\"older norms, by taking $l$ sufficiently large.

\begin{remark} In fact, one can achieve that 
\begin{align*} \| \psi_{k,r} \| \leq C k^{-d}
\end{align*}
for any desired $d$, by increasing $r$. Retracing the argument we see that this will be achieved once $r$ is chosen to satisfy $r > d+3$. Note that this possibly changes the $k_0$ for which this expansion is valid. 
\end{remark}

\section{Results on optimal symplectic connections}\label{sec:uniqueness}

\subsection{Automorphisms and optimal symplectic connections}\label{sec:automs}

This section explains how optimal symplectic connections reflect the geometry of fibrations. Our first result is a variation on the classical Lichnerowicz-Matsushima Theorem. 

The setup requires some notation associated with the submersion $\pi: (X,\alpha) \to (B,\beta)$. Recall that $\Aut_0(\pi)\subset \Aut_0(X)$ denotes the automorphism group of $\pi$. The Lie algebra $\mfh \subset H^0(X,TX^{1,0})$ consists of holomorphic vector fields which vanish somewhere; we denote by $$\mfh_{\pi} \subset \mfh$$ the vector fields whose flow lies in $\Aut_0(\pi)$. $\mfk\subset \mfh$ is the Lie subalgebra of vector fields which correspond to Killing vector fields under the identification $TX^{1,0} \cong TX$; we denote $$\mfk_{\pi} = \mfh_{\pi} \cap \mfk$$ the holomorphic vector fields preserving $\pi$ whose associated real holomorphic vector field is Killing. Note that a vector field $\nu \in \Gamma(X, \scV)$ (with $\scV = \ker d\pi$ the vertical tangent bundle) is Killing with respect to $k g_B + g_X$ for some $k$ if and only if it is Killing for all $k$.

\begin{theorem}\label{fibration-matsushima}
Suppose $\pi: (X,\alpha) \to (B,\beta)$ admits an optimal symplectic connection. Then $$\mfh_{\pi} = \mfk_{\pi} \oplus i \mfk_{\pi}.$$ Thus $\mfh_{\pi}$ is a reductive Lie algebra. 
\end{theorem}

\begin{proof}

Using the notation of Proposition  \ref{rstar-right-invertibility}, we have an identification \begin{align*} \ker p\circ \L_1 &= \ker \scR \cong \mfh_{\pi}, \\ f &\to \nabla_{\scV}^{1,0} f,\end{align*} where we recall the operator $\scR = \bar \partial_B \nabla^{1,0}_{\scV}$. Here $p\circ \L_1 $ is an operator on \emph{complex} valued functions: $$p\circ \L_1: C^{\infty}_E(X,\C) \to C^{\infty}_E(X,\C) .$$ Under this identification, the subspace $\mfk_{\pi}$ corresponds to purely imaginary functions, since $\mfk \subset \mfh$ corresponds to purely imaginary functions. 

By Proposition  \ref{rstar-right-invertibility}, the operator $p\circ \L_1$ is a \emph{real} operator since $\omega_X$ is an optimal symplectic connection. Thus for real functions $u, v$, we have $p\circ \L_1 (u + iv)= 0$ if and only if $$\scR(u) =\scR(v) = 0.$$ Thus $$\mfh_{\pi} = \mfk_{\pi} \oplus J \mfk_{\pi},$$ as claimed.
\end{proof}

\begin{remark}\label{rmk:projective} In the projective setting, so that $X, B$ are projective, $c_1(H) = \alpha, c_1(L) = \beta$, then there is a natural group of automorphisms of the map which linearise to $H$: $$ \Aut_0(\pi, H) = \Aut_0(\pi) \cap \Aut(X,H).$$ In this case, one sees that $$\Lie(\Aut_0(\pi, H) ) \cong \mfh_{\pi},$$ and thus $\Aut_0(\pi,H)$ is a reductive Lie group. Indeed, automorphsims $g \in \Aut_0(\pi)$ linearise to $H$ if and only if they linearise to $kL+H$ for all $k$   (using additive notation for tensor products), which is ample for $k \gg 0$, and thus the identification follows from the usual identification $$\mfh \cong \Lie(\Aut_0(X, kH+L)).$$ \end{remark}

We next consider the isometry group of an optimal symplectic connection, by which we mean \begin{equation}\Isom_0(\pi, \omega_X) = \Aut_0(\pi) \cap \Isom(X, \omega_X),\end{equation} where define $$ \Isom(X, \omega_X) = \{ g \in \Diff(X): g^*\omega_X = \omega_X\}.$$ The notation is slightly unusual since $\omega_X$ may not be positive.

\begin{theorem}\label{thm:autom-invariance}
The isometry group $\Isom_0(\pi, \omega_X) $ is a maximal compact subgroup of $\Aut_0(\pi)$.

\end{theorem}

\begin{proof}

We use the results proved in Section \ref{sec:adiabatic}. Theorem \ref{thm:solutions} proves that if $\omega_X$ is an optimal symplectic connection, the class $k\beta + \alpha$ admits a twisted extremal metric $\tilde \omega_k$ with twist $\pi^*\xi$ for all $k \gg 0$. Since $\xi$ is positive on $B$, we can appeal to the results mentioned in Section \ref{sec:prelims} to relate the geometry of these twisted extremal metrics to the geometry of the map $\pi: X \to B$. 

By Theorem \ref{thm:automorphisms-of-maps}, each $\tilde \omega_k$ is invariant under a maximal compact subgroup of $\Aut_0(\pi)$. By Lemma \ref{lemma-conjugation} proved in the subsequent section, we may assume that these are all the \emph{same} maximal compact subgroup $K \subset \Aut_0(\pi)$. 

Take some $g \in K$, and suppose $$g^*\omega_X = \omega_X + \ddb \phi.$$ We must show $\phi$ is constant. By invariance of twisted extremal metrics, $$g^*\tilde \omega_k = \tilde \omega_k.$$ Since $g \in \Aut_0(\pi)$ preserves base forms, we have $$g^*(k\omega_B + \ddb \psi_{B,2}) = k\omega_B + \ddb\psi_{B,2}.$$ It follows that $$g^*(\tilde \omega_k - (k\omega_B  + \omega_X + \ddb \psi_{B,2})) -  (\tilde \omega_k - (k\omega_B  + \omega_X + \ddb \psi_{B,2}))= \ddb \phi.$$ 

By Theorem  \ref{thm:solutions}, we have $$|\tilde \omega_k - (k\omega_B  + \omega_X + \ddb \psi_{B,2})|_{C^2} \leq C k^{-1}.$$ Thus $\phi$, which is independent of $k$, satisfies $$|\ddb \phi|_{C^2} \leq Ck^{-1},$$ and hence $\ddb \phi$ is zero and $\phi$ is constant.

This shows that $K \subset \Isom_0(\pi, \omega_X)$. Since $K$ is a \emph{maximal} compact subgroup of $\Aut_0(\pi)$, we must have $K = \Isom_0(\pi, \omega_X)$, proving the result.
\end{proof}

We will also later require a more technical result concerning the Lie algebra of the full automorphism group. For manifolds,  a good exposition of such results is contained in \cite[Section 6.3]{darvas-rubinstein}. Denote $\mfg = \Lie \Aut_0(X)$ and $\mfg_{\pi} \subset \mfg = \Lie \Aut_0(\pi)$. Let $\tilde\mfk_{\pi} $ denote the real holomorphic vector fields associated with $\mfk_{\pi}$ via the isomorphism $TX \cong TX^{1,0}$. Denote also $\mfa_{\pi}\subset \mfg_{\pi}$  the Lie subalgebra of harmonic forms $$\mfa_{\pi} = \{\nu \in \mfg_{\pi} : (kg_B + g_X)(\nu, \cdot) \textrm{ is a harmonic } 1\textrm{-form}\},$$ where $g_B, g_X$ are the tensors induced $\omega_B$ and $\omega_X$ respectively. Note $(kg_B + g_X)$ is a Riemannian metric for $k\gg 0$. Since $\nu$ is a \emph{vertical} vector field, we have $g_B(\nu, \cdot)= 0$ as harmonicity of $(kg_B + g_X)(\nu, \cdot)$ is \emph{independent} of $k$. This justifies the notation $\mfa_{\pi}$ omitting $k$.

\begin{corollary}\label{use-for-darvas-rubinstein}

Suppose $\omega_X \in \alpha$ is an optimal symplectic connection. Then \begin{align*} \isom(\pi, \omega_X) &=  \mfa_{\pi} \oplus \tilde\mfk_{\pi}, \\ \mfg &= \mfa_{\pi} \oplus \tilde\mfk_{\pi} \oplus J \tilde\mfk_{\pi},\end{align*} with $\isom(\pi, \omega_X)$ the Lie algebra of $\Isom_0(\pi, \omega_X)$.
\end{corollary}

\begin{proof} 

This follows immediately from Theorem \ref{fibration-matsushima} and \cite[Lemma 2.1.1]{gauduchon}, which decomposes an arbitrary real holomorphic vector field as a sum $$\nu = \nu_{H} + \nabla \phi + J \nabla \psi,$$ with $\nu_H$ dual to a harmonic form and $\phi, \psi$ real-valued functions. This decomposition depends on a choice of K\"ahler metric $k\omega_X + \omega_B$, and its associated Riemannian metric. However as we are considering vertical vector fields, as mentioned above harmonicity with respect to $g_X + kg_B$ is independent of $k$. Similarly, the for vertical vector fields these gradients can be taken to be vertical gradients using $g_X$, giving the result.
\end{proof}

\subsection{Interlude on twisted extremal metrics}

We return to the setting and notation of Section \ref{sec:twisted-extremal}, so $\omega \in \beta$ is a twisted extremal metric on a morphism $q: Y \to W$ between K\"ahler manifolds with $Y$ compact, and with twisting form $\zeta$ on $Y$ the pullback of a K\"ahler metric on $W$. In practice, we we will apply our results to the morphism $\pi: X \to B$. Thus by Theorem \ref{thm:automorphisms-of-maps}, the isometry group $\Isom_0(q)$ is a maximal compact subgroup of $\Aut_0(q)$. We require:

\begin{lemma}\label{lemma-conjugation} Suppose $K$ is a maximal compact subgroup of $\Aut_0(q)$. Then there is a twisted extremal metric $g^*\omega \in \beta$ such that $\Isom_0(q, g^*\omega) = K,$ for some $g \in \Aut_0(\pi)$.
\end{lemma}

\begin{proof} Maximal compact subgroups of $\Aut_0(q)$ are each conjugate to each other, so suppose $$gKg^{-1} = \Isom_0(q).$$ Then $g^*\omega$ is a twisted extremal metric with the same twist (since $g^*\zeta = \zeta$), and moreover the isometry group of $g^*\omega$ is seen to equal $K$.
\end{proof}

We also require a more precise statement concerning the uniqueness of twisted extremal vector fields. We begin by proving that the twisted Futaki invariant is independent of choice of K\"ahler metric; this can be shown to be a consequence of \cite[Corollary 6.9]{kahler}, but it seems worth providing a direct proof not reliant on that more difficult statement. The result in the twisted K\"ahler-Einstein setting is due to Datar-Sz\'ekelyhidi \cite[Proposition 8]{datar-szekelyhidi}, using a different approach.

\begin{proposition}\label{prop:twisted-futaki} Let $\nu \in \mfh_{q}$ have associated holomorphy potential $h$ with respect to $\omega$. Then the twisted Futaki invariant $$F_{\zeta}(\nu) = \int_Y h(S(\omega) - \Lambda_{\omega}\zeta - \hat S_{\zeta})\omega^n$$ is independent of choice of $\omega \in \beta$.\end{proposition}

\begin{proof}

Take a family of K\"ahler metrics $\omega_t \in \beta$, and suppose $\nu$ has associated potentials $h_t $. Setting $$F_{\zeta,t }(\nu) =  \int_Y h_t(S(\omega_t) - \Lambda_{\omega_t}\zeta - \hat S_{\zeta})\omega_t^n,$$ it suffices to show $$\frac{d}{dt}_{\big| t=0}F_{\zeta,t}(\nu) = 0.$$

Write $$\frac{d}{dt}_{\big| t=0} \omega_t = \ddb \phi.$$ Then $$\frac{d}{dt}_{\big| t=0} \omega_t^n = \Delta \phi\omega^n.$$ Recall also that we have
$$\frac{d}{dt}_{\big| t=0} \left( S(\omega_t) - \Lambda_{\omega_t}\zeta \right) = \overline{\scL_{\zeta} (\phi)} - \langle \nabla^{1,0} \left( S(\omega) - \Lambda_{\omega} (\zeta ) \right), \nabla^{1,0} \phi \rangle . $$
This follows from our previous formula \eqref{eq:linearisation} for the linearisation of the twisted scalar curvature, by conjugating and using that the twisted scalar curvature is a real operator. 
Finally, by Lemma \ref{lem:changeinpot}, we also have  
 $$\frac{d}{dt}_{\big| t=0} h_t =  \nu \left( \phi \right).$$
Thus
\begin{align*} \frac{\textnormal{d}}{\textnormal{dt}}_{\big| t=0} \bigg( F_{\zeta, t} (\nu ) \bigg) = \int_Y \bigg(& \nu (\phi) ( S(\omega) - \Lambda_{\omega}\zeta - \hat S_{\upsilon} ) + h (S(\omega)- \Lambda_{\omega}\zeta - \hat S_{\zeta}) \Delta(\phi) \\
&- h \overline{\scL_{\zeta} (\phi)} + h \langle \nabla^{1,0} \left( S(\omega)- \Lambda_{\omega}\zeta \right), \nabla^{1,0} \phi \rangle  \bigg) \omega^n,
\end{align*}

Note the identity $$ \int_Y f \Delta (\phi) \omega^n = - \int_Y \langle \nabla^{1,0} f, \nabla^{1,0} \phi \rangle \omega^n.$$ Applying this to $f= h (S(\omega)- \Lambda_{\omega}\zeta - \hat S_{\zeta})$ gives
\begin{align*} \int_Y  h (S(\omega)- \Lambda_{\omega}\zeta - \hat S_{\zeta}) \Delta (\phi) \omega^n =& - \int_Y \langle \nabla^{1,0}  \left( h (S(\omega)- \Lambda_{\omega}\zeta - \hat S_{\zeta}) \right), \nabla^{1,0} \phi \rangle \omega^n \\
=& - \int_Y h \langle \nabla^{1,0} \left( S(\omega)- \Lambda_{\omega}\zeta \right), \nabla^{1,0} \phi \rangle \omega^n \\
& - \int_Y  \left( S(\omega)- \Lambda_{\omega}\zeta - \hat S_{\zeta} \right) \langle \nabla^{1,0} h, \nabla^{1,0} \phi \rangle \omega^n .
\end{align*}
Moreover, $\langle \nabla^{1,0} h, \nabla^{1,0} \phi \rangle = \nu (\phi).$

Combining the above, we obtain
\begin{align*} \frac{\textnormal{d}}{\textnormal{dt}}_{\big| t=0} \bigg( F_{\zeta, t} (\nu ) \bigg) &= - \int_Y h \overline{\scL_{\zeta} (\phi)}  \omega^n \\
&=  \int_Y \scL_{\zeta} (h)  \phi \omega^n \\
&=0,
\end{align*}
as $h \in \ker \scL_{\zeta}$ by Corollary \ref{kernel-geometrically}.
\end{proof}

This allows us to prove the following, which is a variant of results of Futaki-Mabuchi \cite{futaki-mabuchi}, and is proven in a similar way. Another exposition of such results for extremal metrics is given by Berman-Berndtsson \cite[Section 4.1]{berman-berndtsson}. 

\begin{proposition}\label{unique-vector-field}
Suppose $\omega, \omega' \in \beta$ are twisted extremal metrics with twisted extremal vector fields $\nu, \nu'$ respectively such that $\Isom_0(q, \omega) = \Isom_0(q, \omega')$. Then $\nu = \nu'$.
\end{proposition}

\begin{proof}

Recall Futaki-Mabuchi's bilinear form on holomorphic vector fields $\nu, \nu'$ with mean-value zero holomorphy potentials $h_{\nu}, h_{\nu'}$ with respect to $\omega$ given by $$\langle \nu, \nu' \rangle = \int_Y h_{\nu} h_{\nu'} \omega^n.$$ This is independent of choice of $\omega \in \beta$ by Futaki-Mabuchi \cite{futaki-mabuchi}. 

Denote by $\mfh_{q}^K\subset \mfh_{q}$ the subspace of holomorphic vector fields $\nu$ such that the flow of the imaginary part $\Ima \nu$ lies in $K$. Then just as in Futaki-Mabuchi \cite{futaki-mabuchi} or \cite[Proposition 4.11]{berman-berndtsson}, the bilinear form $\langle \cdot, \cdot \rangle$ is real-valued and positive-definite on $\mfh_q^K$. Indeed, the fact that the pairing is real-valued  is a consequence of  their claim that $h_{\nu}$ is real-valued, which follows since $\omega$ is invariant under the flow of $\Ima \nu$ \cite[p. 1189]{berman-berndtsson}. Positive-definiteness is a simple consequence of the holomorphy potentials being real-valued, as the $L^2$-inner product is certainly positive-definite. In particular, the inner-product is non-degenerate. 

Returning to the claim of the Proposition, we wish to show the two twisted extremal vector fields are equal. We first claim $\nu, \nu' \in  \mfh_{q}^K$, which follows from the fact that the flow of $\Ima \nu$ is an isometry and our hypothesis that $K = \Isom_0(\pi, \omega) = \Isom_0(\pi, \omega')$. To see that  the flow of $\Ima \nu$ is an isometry, note that \begin{equation}\label{eqn:twisted-extremal-vector-field}\nu = \nabla^{1,0}(S(\omega) - \Lambda_{\omega}\zeta),\end{equation} is the ($(1,0)$-part of) the gradient of a \emph{real} valued function, which implies the claim \cite[Proof of Proposition 4.14]{berman-berndtsson}. Thus by non-degeneracy of the inner-product, it is enough to show $$\langle \tau, \nu \rangle = \langle \tau, \nu'\rangle $$ for all $\tau \in \mfh^K_{q}$.

Denote by $h_{\tau}$ and $h_{\tau}'$ the holomorphy potential of $\tau$ of mean-value zero with respect to $\omega$ and $\omega'$ respectively. Then since $\nu$  is the twisted extremal vector field given by Equation \eqref{eqn:twisted-extremal-vector-field}, by definition of the inner product we have $$\langle \tau, \nu\rangle = \int_Y h_{\tau}(S(\omega) - \Lambda_{\omega}\eta - \hat S_{\zeta})\omega^n.$$ This is, by definition, the twisted Futaki invariant of $\tau$, which by Proposition \ref{prop:twisted-futaki} is independent of choice of $\omega \in \alpha$. Thus $$\int_Y h_{\tau}(S(\omega) - \Lambda_{\omega}\zeta - \hat S_{\zeta})\omega^n = \int_Y h_{\tau}'(S(\omega') - \Lambda_{\omega'}\zeta - \hat S_{\zeta})\omega'^n,$$ which means $$\langle \tau, \nu \rangle = \langle \tau, \nu'\rangle,$$ which proves the result.
\end{proof}

\subsection{Uniqueness of optimal symplectic connections}

We begin with the discrete automorphism group case.

\begin{theorem} Suppose $\Aut_0(\pi)$ is trivial. If $\omega_X, \omega_X' \in \alpha$ are two optimal symplectic connections, then $$\omega_X = \omega_{X}' + \pi^* \ddb \phi_B,$$ with $\phi_B\in C^{\infty}(B,\R)$.\end{theorem}

\begin{proof}

From $\omega_X$ and $\omega_X'$, Theorem \ref{thm:solutions} constructs twisted cscK metrics $\tilde \omega_k$ and $\tilde \omega_k'$ respectively in the K\"ahler class $k\alpha + \beta$. These are twisted cscK rather than twisted extremal, since $\Aut_0(\pi)$ and hence $\mfh_{\pi}$ are trivial. Since $\pi: X \to B$ has no continuous automorphisms, uniqueness of twisted cscK metrics described in Theorem \ref{thm:automorphisms-of-maps} implies $$\tilde \omega_k = \tilde \omega_k'.$$ By Theorem \ref{thm:solutions}, we have \begin{align*}|\tilde \omega_k - (k\omega_B + \omega_X + \ddb \phi_{B,2})|_{C^2}&\leq Ck^{-1}, \\ |\tilde \omega_k' - (k\omega_B + \omega'_X + \ddb \phi'_{B,2})|_{C^2}&\leq C'k^{-1}.\end{align*} Thus by the triangle inequality $$|\omega_X + \ddb \phi_{B,2} - \omega_X' - \ddb \phi'_{B,2}|_{C^2}\leq (C+C')k^{-1},$$ and hence this quantity is independent of $k$ we see that $$\omega_X + \ddb \phi_{B,2} = \omega_X' + \ddb \phi'_{B,2}.$$ Hence $$\omega_X = \omega_X' + \ddb (\phi'_{B,2}-\phi'_{B,2}),$$ as required.
\end{proof}

\begin{remark} As in Remark \ref{rmk:projective}, this result can be slightly sharpened and clarified when $X$ and $B$ are projective, with $\alpha = c_1(H)$ and $\beta = c_1(L)$. Then all that is required is that $\Aut_0(\pi, H)$ is discrete.\end{remark}

The case when $\Aut_0(\pi)$ is non-trivial is more challenging.  We first prove a weak version of uniqueness. Denoting $G = \Aut_0(\pi)$, we write $\omega_X \in_B  \overline{G.\omega_X'}$ if there exist elements $g_t \in G$ such that $$\lim_{t \to 0}|\omega_X - g_t^*\omega_X'|_{C^2} = 0.$$

\begin{proposition}\label{weak-uniqueness} If $\omega_X, \omega_X' \in \alpha$ are two optimal symplectic connections, then there is a function $\phi_B\in C^{\infty}(B,\R)$ such that $$\omega_X \in  \overline{G.(\omega_X'+ \ddb \phi_B)}.$$ \end{proposition}

\begin{proof} To $\omega_X$ and $\omega_X'$,  Theorem \ref{thm:solutions} associates twisted extremal metrics  $\tilde \omega_k, \tilde \omega_k' \in k\alpha + \beta$ respectively. By Lemma \ref{lemma-conjugation}, there is a $g_k \in  \Aut_0(X/B)$ such that $\tilde \omega_k$ and $g_k^*\omega_k'$ have the same isometry group $$\Isom_0(\pi, \tilde \omega_k) = \Isom_0(\pi, g_k^*\tilde \omega_k').$$ Thus by Proposition \ref{unique-vector-field}, the twisted extremal vector fields associated to $\tilde \omega_k$ and $g_k^*\tilde \omega_k'$ are actually equal. Since these metrics have the same associated vector field, Theorem \ref{thm:automorphisms-of-maps} produces $h_k\in \Aut_0(\pi)$ such that \begin{equation}\label{equality-eqn}\tilde \omega_k = (h_k \circ g_k)^*\tilde \omega_k'.\end{equation}

We can now argue as before. Theorem \ref{thm:solutions}  implies that \begin{align}\label{align-label}|\tilde \omega_k - (k\omega_B + \omega_X + \ddb \phi_{B,2})|_{C^2}&\leq Ck^{-1}, \\ |\tilde \omega_k' - (k\omega_B + \omega'_X + \ddb \phi'_{B,2})|_{C^2}&\leq C'k^{-1},\end{align} with $\phi_{B,2},\phi_{B,2}' \in C^{\infty}(B,\R)$.  We claim that $$\omega_X \in \overline{G.(\omega_X'+ \ddb (\phi_{B,2}' - \phi_{B,2}))}.$$ Indeed, Equations \eqref{equality-eqn}, \eqref{align-label} and the triangle inequality show that $$|\omega_X -  (h_k \circ g_k)^*(\omega_X'+\ddb (\phi_{B,2}' - \phi_{B,2}))|_C^0 \leq (C+C')k^{-1},$$ meaning $$\lim_{k \to \infty} |\omega_X -  (h_k \circ g_k)^*(\omega_X'+\ddb (\phi_{B,2}' - \phi_{B,2}))|_{C^2} = 0,$$ as required.
\end{proof}

It is not straightforward to conclude from this that $\omega_X \in G.(\omega_X'+ \ddb \phi_B)$. In fact, being able to pass from two K\"ahler metrics being in the same orbit \emph{closure} to the same orbit  in certain circumstances was one of the key new ingredients in the breakthrough of Darvas-Rubinstein \cite{darvas-rubinstein}. We use their techniques to conclude the uniqueness result we desire. 

Fix some $k \gg 0$ such that $k\beta + \alpha$ admits a twisted extremal metric, and let $\omega_k = k\omega_B + \omega_X$ be a reference metric. Let $\H_k$ be the space of K\"ahler potentials with respect to $\omega_k$.

\begin{definition} For $t \in [0,1]$, let $t \to \alpha_t \in \H_k$ be a smooth curve joining K\"ahler potentials $u_0, u_1 \in \H_k$. We define the \emph{length} of $\alpha$ to be $$\ell_1(\alpha) = \int_0^1 \|\dot \alpha_t\|_{\omega_k+\ddb \alpha_t} dt,$$ where $$\|\phi \|_{\omega_k+\ddb \alpha_t} = V^{-1}\int_X |\phi|(\omega_k+\ddb \alpha_t)^{m+n}$$ and $V = \int_X \omega_k^{m+n}$. The $d_1$\emph{-metric} is defined by $$d_1(u_0,u_1) = \inf \{ \ell_1(\alpha): \alpha \textrm{ is a smooth curve with } \alpha(0) = u_0, \alpha(1) = u_1\}.$$
\end{definition}

This is justified by the following result of Darvas.

\begin{theorem}\label{darvas-metric}\cite[Theorem 3.5]{darvas} $d_1$ is a metric on $\H_k$. \end{theorem}

In the presence of automorphisms, for any $G \subset \Aut_0(X),$ the natural pseudo-metric is defined by Darvas-Rubinstein to be \cite[Section 2]{darvas-rubinstein} \begin{align*} &d_{1,G}: \H_k \times \H_k \to \R, \\ &d_{1,G}(\phi,\psi) = \inf_{g \in G} d_1(\phi, g.\psi),\end{align*} where by definition $$\omega_k - g^*(\omega_k + \ddb \psi) = \ddb g.\psi.$$  We will always take $G = \Aut_0(\pi)$.

\begin{remark} Both $d_1$ and $d_{1,G}$ depend on $k$; since we have fixed $k$, we omit this in the notation. \end{remark}

We require the following bound, due to Darvas \cite[Corollary 4.14, Theorem 3]{darvas}.

\begin{proposition}\label{darvas-bound} There is a constant $C>1$ such that $$d_1(\phi,\psi) \leq C\int_X | \phi - \psi| (\omega_{k} + \ddb \phi )^{m+n} + C \int_X  | \phi - \psi| (\omega_{k} + \ddb \psi )^{m+n} .$$\end{proposition}

Using this, we can understand orbit closures via the $d_{1,G}$-pseudometric.

\begin{lemma}\label{closed-orbit-to-d1} Suppose $\omega_X \in  \overline{G.\omega_X'}$ and $$\omega_X' = \omega_X + \ddb \phi.$$ Then $$d_{1,G}(0,\phi) = 0.$$  \end{lemma}

\begin{proof} Firstly, remark that since $$\omega_X' = \omega_X + \ddb \phi,$$ we also have $$k\omega_ B + \omega_X' = k\omega_B + \omega_X' + \ddb \phi.$$   Since our reference metric is $k\omega_ B + \omega_X$, it follows that $\phi$ is the K\"ahler potential for $k\omega_B + \omega_X'$ with respect to $\omega_k = k\omega_B + \omega_X$.

Since $ \omega_X \in  \overline{G.\omega_X'},$ there is a sequence $g_t \in G$ (with, say, $g_1 = \Id$ so that our notation is consistent) such that writing $$g_t^*\omega_X' - \omega_X = \ddb \phi_t,$$ we have $$\lim_{t \to 0}\|\phi_t\|_{C^{0}} = 0.$$ By definition, we have $g_t.\phi = \phi_t$.  Proposition \ref{darvas-bound} then gives \begin{align*} d_1(0,g_t.\phi) &\leq C\int_X | g_t.\phi | (\omega_{k} )^{m+n} + C \int_X  |g_t.\phi| (\omega_{k} + \ddb g_t.\phi )^{m+n}, \\ & \leq CV \|g_t.\phi \|_{C^0},\end{align*} where as above $$V = \int_X \omega_k^{m+n} = \int_X (\omega_{k} + \ddb g_t.\phi )^{m+n}.$$ Thus since $\|\phi_t\|_{C^{0}} \to 0$, we have $$\lim_{t \to 0} d_1(0,g_t.\phi) = 0,$$ and hence by definition $d_{1,G}(0,\phi) = 0$.
\end{proof}

The assumption of this Lemma is that $\omega_X \in  \overline{G.\omega_X'}$ and $\omega_X' = \omega_X + \ddb \phi,$ and the conclusion is that $d_{1,G}(0,\phi) = 0$. In order to pass from this to the desired conclusion, we need to show that in this situation there is a $g\in G$ with $g^*\omega_X' = \omega_X$. This is precisely the kind of statement considered by Darvas-Rubinstein \cite[Property (P6)]{darvas-rubinstein}. We continue with the notation of Lemma \ref{closed-orbit-to-d1}.

\begin{proposition}\label{darvas-rubinstein-analogue} Suppose $d_{1,G}(0,\phi) = 0.$ Then there is a $g \in G$ such that $$d_{1,G}(0,\phi) = d_{1}(0,g.\phi).$$\end{proposition}

\begin{proof} This follows from work of Darvas-Rubinstein, which requires some notation to set up. 

Let $K$ be the isometry group of $\omega_X$, so that $K$ is a maximal compact subgroup of $G = \Aut_0(\pi)$. We begin with the Lie algebras. Corollary \ref{use-for-darvas-rubinstein} produces decompositions of the Lie algebras $\isom(\pi,\omega_X) = \Lie K$ and $\mfg_{\pi} = \Lie G$ of the form \begin{align*} \isom(\pi,\omega_X) &= \mfa_{\pi} \oplus \tilde  \mfk_{\pi}, \\ \mfg_{\pi} &= \mfa_{\pi} \oplus \tilde \mfk_{\pi} \oplus J\tilde \mfk_{\pi},\end{align*} with $\mfa_{\pi}$ a Lie subalgebra of the centre of $ \tilde \mfk_{\pi}$ (since $\mfa$ is contained in the centre of $\tilde \mfk$), and a Lie subalgebra $\tilde \mfk_{\pi} \subset \isom(\pi,\omega_X) $. These are the hypotheses of \cite[Proposition 6.2]{darvas-rubinstein}, which allows one to conclude a surjectivity property of the exponential map.

 Next, we claim we have the following:
\begin{enumerate}
\item a metric $d_1$ on $\H_k$;
\item an action of $G$ on $\H_k$ which is a $d_1$-isometry;
\item an element $\omega_k$ which satisfies $K.\omega_k = \omega_k$. 
\item for each $\tau \in \tilde \mfk_{\pi},$ $$t \to \exp(tJ\tau).\omega_k$$ is a $d_1$-geodesic  whose speed depends continuously on $\tau$;
\item for all $\phi, \psi \in \H_k$, the map \begin{align*} G &\times G \to \R, \\ (f,g) &\to d_1(f.\phi,g.\psi)\end{align*} is continuous. 
\end{enumerate}

Property $(i)$ follows from Darvas's Theorem \ref{darvas-metric}. The second property, namely that the $G$-action is a $d_1$-isometry follows from \cite[Lemma 5.9]{darvas-rubinstein}. The third property follows from Lemma \ref{thm:autom-invariance}, since $g^*(k\omega_B + \omega_X) = k\omega_B = \omega_X$ when $g \in K \subset \Aut_0(\pi)$. Item $(iv)$ follows from an identical argument to as in \cite[Proof of Theorem 7.1 (P6)]{darvas-rubinstein}, namely a combination of Mabuchi's classical result that vector fields induce geodesics in the space of K\"ahler potentials \cite[Theorem 3.5]{mabuchi} and the explicit formula for the $d_1$-metric \cite[Theorem 4.3]{darvas-rubinstein} to obtain the continuity of speed. The final item, namely $(vi)$, follows from an argument using Green's functions \cite[Proof of Theorem 7.1 (P6)]{darvas-rubinstein}.

With all of this in place, the hypotheses of  \cite[Proposition 6.8]{darvas-rubinstein} are satisfied and this result provides the desired conclusion: there is a  there is a $g \in G$ with $$d_{1,G}(0,\phi) = d_{1}(0,g.\phi).$$
\end{proof}

\begin{theorem}  If $\omega_X, \omega_X' \in \alpha$ are two optimal symplectic connections, then $$\omega_X = g^*\omega_{X}' + \pi^* \ddb \phi_B,$$ with $g \in \Aut_0(\pi)$ and $\phi_B\in C^{\infty}(B,\R)$. \end{theorem}

\begin{proof}

This is immediate from combining the above. Explicitly, by using uniqueness of twisted extremal metrics Proposition \ref{weak-uniqueness} produces a function $\phi_B \in C^{\infty}(B,\R)$ such that $$\omega_X \in  \overline{G.(\omega_X'+ \ddb \phi_B)},$$ where $G=\Aut_0(\pi)$. Set $\omega_X'' = \omega_X' + \ddb \phi_B,$ and let $$\omega_X'' = \omega_X + \ddb \psi.$$ Then $\omega_X  \in  \overline{G.(\omega_X'')}$ and Lemma \ref{closed-orbit-to-d1} implies $$d_{1,G}(0,\psi) = 0.$$ From this, Proposition \ref{darvas-rubinstein-analogue} gives a $g \in G$ with $$d_{1,G}(0,\psi) = d_{1}(0,g.\psi).$$ Since $d_1$ is a metric by Theorem \ref{darvas-metric}, this implies $g.\psi = 0$, which is to say $$\omega_X = g^*\omega_{X}''.$$ Thus $$\omega_X = g^*( \omega_X' + \ddb \phi_B) =  g^*\omega_X' + \ddb \phi_B,$$ proving the result.
\end{proof}

\subsection{The energy functional}\label{sec:energy}

Fix a reference relatively K\"ahler metric $\omega_X \in \alpha.$ By the the $\partial \bar \partial$-Lemma, any other relatively K\"ahler metric in $\alpha$ is of the form $\omega_X + \ddb \phi$ for some function $\phi \in C^{\infty}(X,\R).$ Denote by $\scK_E \subset C^{\infty}(X,\R)$ the set of functions $\phi\in C^{\infty}(X,\R)$ such that $\omega_X + \ddc \phi$ is a relatively cscK metric. 

Any such a relatively cscK metric $\omega_{\phi}$ induces a splitting $$TX = \scV \oplus \scH_{\phi},$$ and also vertical Laplace operator $\Delta_{\scV_{\phi}}$. Denote by $\theta_{\phi}$ the curvature term $$\theta_{\phi} = \Lambda_{\omega_B} \Delta_{\scV_{\phi}} (\Lambda_{\omega_B}\mu^*F_{\scH_{\phi}}) + \Lambda_{\omega_B} \rho_{\scH_{\phi}}$$ involved in the definition of an optimal symplectic connection, and denote by $$p_{\phi}: C^{\infty}(X,\R) \to C^{\infty}_{E_{\phi}}(X)$$ the $\phi$-dependent projection to the space of fibrewise holomorphy potentials with respect to $\omega_{\phi}$. If  $\omega_t = \omega_X + \ddc \phi_t$ is a one-parameter family of relatively cscK metrics, we will denote $\theta_t = \theta_{\phi_t}$, and similarly for the projection operator.

%Note that while $\scV = \ker d\pi$ is a fixed subbundle of $TX$, the vertical Laplace operator depends on $\phi$ through the dependence of the Laplacian on $\phi$.

\begin{definition} Let $\phi_t: [0,1] \to  \scK_E $ be a path, with $\phi_t = 0$ and $\phi_1 = \phi$. We define the \emph{energy functional} (or \emph{log norm functional}) \begin{align*} \scN: \scK_E &\to \R, \\ \phi &\to \scN(\phi),\end{align*}  by
\begin{align}\label{eqn:energyfnal} \frac{d}{dt} \left( \scN (\phi_t ) \right) = -\int_X \dot \phi_t p_t(\theta_t)  \omega_t^m \wedge \omega_B^n,
\end{align} with $\scN(0) = 0.$
\end{definition}

\begin{remark} The functional $\scN$ can be seen as a generalisation of Donaldson's functional in the Hermite-Einstein problem, and the properties we prove here are generalisations of properties of that functional proven by Donaldson. Recall that for $V$ a holomorphic vector bundle over $(B,\beta)$, Donaldson's functional is defined by its variation along a path $h_t$ of Hermitian metrics \cite[p. 11]{donaldson-hermite-einstein}\begin{equation*}\frac{d}{dt}|_{t=0}\scD(h_t) = i \int_B \tr(h_0^{-1}\dot h(0)( \Lambda_{\omega_B}F_{h_0} - \lambda \Id))\omega_B^n,\end{equation*}where $\lambda$ is the appropriate topological constant. Note that \begin{equation}\label{change-metric}h_0^{-1}\dot h(0) = \frac{d}{dt}|_{t=0}(\log(h_t)).\end{equation} The path of Hermitian metrics $h_t$ induces a path of relatively K\"ahler metrics $\omega_{t}$ which restrict to a Fubini-Study metric on each fibre, in such a way that $$\mu_t^*(\Lambda_{\omega_B}F_{h_t} - \lambda \Id) = p_t(\theta_t),$$ with $p_t(\theta_t)$ as above \cite[Proposition 3.17]{morefibrations} and with $\mu_t^*$ the fibrewise comoment map defined using $\omega_t$. A calculation similar to \cite[Proposition 3.17]{morefibrations} using Equation \eqref{change-metric} to calculate the change in K\"ahler potential associated to the $\omega_t$ then shows that, up to a dimensional constant $$\frac{d}{dt}|_{t=0}\scD(h_t) = \frac{d}{dt} \left( \scN (\phi_t ) \right).$$ Thus the two functionals themselves must also agree. \end{remark}

In order for this to be a reasonable definition, we must show that that $\scN$ is well-defined, independent of choice of path, and that any two relatively cscK metrics can be joined by a path of relatively cscK metrics. It will then follow that the functional is independent of adding a constant to $\phi$, so can be viewed as a functional on relatively cscK K\"ahler \emph{metrics} rather than potentials. We begin by showing path connectedness, and will prove in Lemma \ref{lemma:independence} that $\scN$ is indeed well-defined.

\begin{proposition} $\scK_E$ is path-connected.
\end{proposition}
\begin{proof} Suppose that $\omega_X$ and $\omega'_X$ are two relatively cscK metrics. We note first that it suffices to find a family of relatively cscK metrics $\omega_t$ from $\omega_0 = \omega_X$ to an $\omega_1$ whose restriction to each fibre equals $\omega_X'$. Indeed, if $\omega_X$ and $\omega_X'$ have the same restriction to each fibre, then $\omega_X' = \omega_X + \ddb \pi^* \phi$ for some $\phi : B \to \R$, and then $\omega_t = \omega_X + t \ddb \pi^* \phi$ gives the desired path from $\omega_X$ to $\omega_X'$. 

%In the case when $\omega_X$ and $\omega_X'$ differ on fibres, 

We again use the uniqueness of the cscK metrics. Denote  $\omega_b = \omega|_{X_b}$ and $\omega_b' =  \omega_{X'} |_{X_b}$, and let $K_b$ be the identity component of the isometry group of $\omega_b$. By \cite[Proposition 6.2]{darvas-rubinstein}, the map $$C_b: K_b \times \tilde \mfk_b \to \Aut_0(X_b)$$ given by $$C(k, v) = k\exp(Jv)$$ is surjective. Here the notation $\tilde \mfk_b$ denotes the real holomorphic vector fields corresponding to $\mfk_b = \Lie(K_b)$, following Section \ref{sec:automs}. 

Since $\omega_b, \omega_b'$ are cscK, by uniqueness of cscK metrics there is a $f_b: X_b \to X_b$ such that $$\omega'_b = f_b^* \omega_b.$$ By the surjectivity just mentioned, we can write $$f_b = k_b \exp(J v_b),$$ and since $k_b \in K_b$ is an isometry of $\omega_b$, it follows that $$\omega'_b = \exp(J v_b)^* \omega_b.$$ Thus $$\exp(tJ v_b)^* \omega_b$$ is a path of cscK metrics joining $\omega_b$ at $t=0$ to $\omega_b'$ at $t=1$. Since the exponential map is surjective and locally a diffeomorphism for each $b \in B$, one can choose $v_b$ smoothly, hence forming a vector field $v \in \Gamma (\scV)$ whose restriction to each fibre is real holomorphic.

Let $\rho(t)$ be the flow of $v$, so that $\pi \circ \rho = \pi$ as $v$ is a vertical vector field. The form $\rho(t)^* \omega_X$ is thus a smooth form whose restriction to each fibre $X_b$ is a cscK metric (in particular, its restriction to each fibre is a $(1,1)$-form even though $\rho(t)^* \omega_X$ may not necessarily be a $(1,1)$-form itself on $X$). Let $\phi_{t,b}$ be the unique function defined by $$\omega_b -  \rho(t)^* \omega_X)|_{X_b} = \ddbar \phi_{t,b}$$ with $\phi_{t,b}$ normalised to have integral zero over $X_b$. Then the $\phi_{t,b}$ induce a smooth function $\phi_t$ for $t \in [0,1]$ such that $\omega_X + \ddbar \phi_t$ is cscK for all $t \in [0,1]$ and such that at $t=1$ its restriction to each fibre agrees with $\omega_b',$ which is what we wanted to show.
\end{proof}

The following will allow us to simplify our definition of $\scN$.

\begin{lemma}\label{lem:dotphi} Let $\omega_t = \omega_X + \ddc \phi_t$ be a path of relatively cscK metrics. Then $\dot \phi_t$ is a fibrewise holomorphy potential with respect to $\omega_t$.
\end{lemma}
\begin{proof} This is a fibrewise statement. By uniqueness of cscK metrics, on each  fibre $X_b$ there is a $f_{t,b}\in \Aut_0(X_b)$ such that 
$$ \omega_t|_{X_b} = f_{t,b}^* \omega_b,$$ where as usual $\omega_b = \omega_{X| X_b}$. 
Thus the statement reduces to the well-known fact that if $\omega + \ddb \psi_t$ is a path of K\"ahler potentials on a compact K\"ahler manifold obtained from pulling back via automorphisms, then $\dot \psi_t$ is a holomorphy potential with respect to $\omega + \ddb \psi_t$; see, for example,  \cite[Example 4.26]{gabor-book}.
\end{proof}

\begin{corollary}\label{corollary-orthogonality}Let $\omega_t = \omega_X + \ddc \phi_t$ be a path of relatively cscK metrics. Then $$\int_X \dot \phi_t p_t(\theta_t)  \omega_t^m \wedge \omega_B^n=  \int_X \dot \phi_t (\theta_t + \Lambda_{\omega_B} \zeta) \omega_t^m \wedge \omega_B^n,$$ where $\zeta$ is the Weil-Petersson metric on $B$ of Theorem  \ref{weil-petersson-form}. Thus if $\int_{X/B}\dot \phi_t\omega_t^m = 0,$ then $$\int_X \dot \phi_t p_t(\theta_t)  \omega_t^m \wedge \omega_B^n=  \int_X \dot \phi_t \theta_t \omega_t^m \wedge \omega_B^n.$$

\end{corollary}

\begin{proof} 

The $C^{\infty}(B)$-component of $\theta_t$ is $-\Lambda_{\omega_B} \zeta$ (which is independent of $t$), which is simply to say that $$\int_{X/B}( \theta_t + \Lambda_{\omega_B} \zeta)\omega_t^m = 0.$$  The above Lemma proves that the $C^{\infty}_{R_t}(X)$-component of $\dot \phi_t$ vanishes, where $C^{\infty}_{R_t}(X)$ consists of functions of fibre integral zero with respect to $\omega_t$ and orthogonal to holomorphy potentials. This proves the claims by the definition of the projection $p_t$.
%if we write $\dot \phi_t = \dot \phi_t^{\scV} + \dot \phi_t^{\scH},$ with respect to the horizontal-vertical decomposition determined by $\omega_t$, then $\dot \phi_t^{\scV} \in C^{\infty}_{E_t} (X),$ where $E_t = E_{\omega_t}$ consists of the fibrewise average $0$ holomorphy potentials for $\omega_t$. 
\end{proof}

The main goal of the Section is to prove the following Theorem. Fix a reference relatively cscK metric $\omega_X$, and suppose $ \omega_X + \ddb \psi$ is an optimal symplectic connection. Denote by  $K=\Isom_0 (\pi , \omega+\ddb \psi)$ the isometry group of the optimal symplectic connection, which is a maximal compact subgroup of $\Aut_0(\pi)$ by Theorem \ref{thm:autom-invariance}. Let $\scK_E^K$ denote the space of potentials for relatively cscK which are $K$-invariant, so that when $\Aut_0(\pi)$ is trivial, this is our usual space $\scK_E$.

\begin{theorem}\label{thm:energy} Suppose that $\pi: (X, \alpha) \to (B,\beta)$ admits an optimal symplectic connection $\omega_X+ \ddb \psi$. Then for all $\psi \in \scK_E^K$, we have $$\scN(\phi) \geq \scN(\psi).$$ Thus $\scN$ is bounded below. \end{theorem}

This will  follow our general strategy of perturbing to known results for twisted extremal metrics. For all $k \gg 0$, associated to $\omega_X$ is a family of twisted extremal metrics $\tilde \omega_k \in k\beta + \alpha$, with isometry group $K= \Isom(\pi,\omega_X+\ddb \psi)$ equal to that of $\omega_X+\ddb \psi$ itself. As before, the twist is $\pi^*\xi$, where $\xi$ is positive on $B$; we will abuse notation by omitting pullbacks. Denote by $\scK^K_k$ the space of $K$-invariant K\"ahler potentials in the class $ k\beta + \alpha$ with respect to the reference metric $\omega_k = k\omega_B + \omega_X$. 

Suppose the extremal vector field $v_k$ associated to $\tilde \omega_k$ has holomorphy potential $h_k$ with respect to the reference metric $k\omega_B + \omega_X$, normalised so that $h_t$ has integral zero with respect to the volume form $\omega_k^{m+n}$. For a family $\omega_t = \omega_k+ \ddb \phi_t$ of K\"ahler metrics invariant under $K$, let $h_{t,k}$ denote the corresponding holomorphy potential for $v_k$ of integral zero with respect to $\omega_t$. 

\begin{definition}\label{defn:twistedmabuchi} The \emph{modified twisted Mabuchi functional} $\scM_{\xi,k} : \scK^K_k \to \R$ is defined by its variation 
$$ \frac{d}{dt} \left( \scM_{\xi,k} (\phi_t) \right) = -\int_X \dot \phi_t \left( S(\omega_{t,k}) - \Lambda_{\omega_t} (\xi) - h_{t,k} - \hat S_{\xi,k\beta + \alpha} \right) \omega_{t,k}^n $$
along a path $\omega_{t,k} = \omega_k + \ddb \phi_t$ of K\"ahler metrics, where $\hat S_{\xi,k\beta + \alpha}$ is the average twisted scalar curvature of any metric in $\beta$.
\end{definition}

We will rely on a result of Berman-Berndtsson, which states that a twisted extremal metric achieves the absolute minimum of the modified twisted Mabuchi functional. Their proof does not rely on the geometry of our situation, and instead applies when $\nu$ is a semi-positive form on $X$.

\begin{theorem}\label{thm:lowerboundmabuchi}The functional $\scM_{\xi,k} : \scK^K_k \to \R$ is bounded below by $\scM_{\xi,k} (\tilde \omega_k)$. 
\end{theorem}

\begin{proof} This follows directly from the techniques of Berman-Berndtsson, who prove this statement when $\tilde \omega_k$ is cscK \cite[Corollary 1.2]{berman-berndtsson}. Their argument relies on \emph{convexity} of the Mabuchi functional along weak geodesics. They also prove convexity of the modified Mabuchi functional along potentials invariant under the flow of the extremal vector field \cite[Section 4.1]{berman-berndtsson}. Note that extremal vector field lies in the isometry group of the extremal metric. Thus their approach also proves that an extremal metric achieves the absolute minimum of the modified Mabuchi functional on $K$-invariant metrics. 

Berman-Berndtsson also consider the twisted Mabuchi functional \cite[Section 3.1.1]{berman-berndtsson}, where they show that the term added to the Mabuchi functional is strictly convex when the twisting form is positive. It is easy to see that the added term is convex when the twisting form is semipositive. Thus comibing their results proves convexity of the modified twisted Mabuchi functional on the space of $K$-invariant metrics, and hence their approach proves that a twisted extremal metric achieves the absolute minimum of the modified twisted Mabuchi functional on this space. 
\end{proof}

\begin{remark}\label{rem:csckmabuchibound} When the metrics are actually twisted cscK rather than twisted extremal, the $K$-invariance assumption is not actually needed: the functional $\scM_{\xi,k}$ is then bounded below on all of $\scK_k$.
\end{remark}

The key asymptotic expansion we will use is the content of the following Lemma. 

\begin{lemma}\label{lem:energyasymptotic} We have $$ \scM_{\xi,k} (\phi) - \scM_{\xi,k} (\psi) = {m+n \choose n} k^{n-1} \left( \scN (\phi) - \scN (\psi) \right) + O(k^{n-2}).$$

\end{lemma} 
\begin{proof}

Let $\omega_t = \omega_X + \ddc \phi_t \in \alpha$ be a path of relatively cscK metrics on $\pi: X \to B$, with $\phi_0 = \psi$ and $\phi_1 = \phi$. Denote $\omega_k = \omega_X + k \omega_B$ and $\omega_{t,k} = \omega_k + \ddb \phi_t$. Thus $$\int_0^1  \left( \frac{d}{dt} \scN (\phi_t) \right)dt =  \scN (\phi) - \scN (\psi),$$ and similarly for $\scM_{\xi,k} (\phi) - \scM_{\xi,k} (\psi)$.

The derivative of the modified twisted Mabuchi functional $\scM_{\xi,k}$ along the path $\omega_{t,k}$ is
$$ \frac{d}{dt} \left( \scM_{\xi} (\phi_t ) \right) =- \int_X \dot \phi_t \left( S(\omega_{t,k}) - \Lambda_{\omega_{t,k}} \xi - h_{t,k} - \hat S_{\xi,k} \right) \omega_{t,k}^{m+n}.$$
We first note that by Corollary \ref{cor:R1-kill}, the holomorphy potential $h_{k,t}$ for the extremal vector field is in fact $O(k^{-2})$. Thus 
$$ \frac{d}{dt} \left( \scM_{\xi} (\phi_t ) \right) =- \int_X \dot \phi_t \left( S(\omega_{t,k}) - \Lambda_{\omega_{t,k}} \xi  - \hat S_{\xi,k} \right) \omega_{t,k}^{m+n} + O(k^{n-2}).$$

Denote by $ S(\omega_{t,b})$ the function whose restriction to each fibre $X_b$ is the scalar curvature of $\omega_{t | X_b}$, and let  $\hat S_b$ be the average scalar curvature on the fibres $(X_b, \omega_{t,b})$, which is independent of both $t\in [0,1]$ and $b \in B$. By our hypothesis that $\omega_X$ is relatively cscK, we have $S(\omega_{t,b}) = \hat S_b$. Recall from Proposition \ref{prop:expansion-scalar} that the twisted scalar curvature (with twist $\xi$) expands as 
$$S(\omega_{t,k}) - \Lambda_{\omega_{t,k}} (\xi) = S(\omega_{t,b}) + k^{-1} \left( \lambda_{R_t} + p(\theta_t) + S(\omega_B) - \Lambda_{\omega_B} (\xi + \zeta) \right) + O(k^{-2}),$$
where $\lambda_{R_t} \in C^{\infty}_{R_t}(X)$, since $\xi$ is pulled back from $B$. It is important to note here that the Weil-Petersson form $\zeta$ is independent of $t$, as it is independent of choice of relatively cscK metric. The volume form has an expansion
$$ \omega_t^{m+n} = {m+n \choose n} k^n \omega_t^m \wedge \omega_B^n + O(k^{n-1}).$$
The average scalar curvature expands as 
$$ \hat S_{\xi,k} = \hat S_b + k^{-1} \hat S_{\xi+\zeta} + O(k^{-2}),$$
where $\hat S_{\xi + \zeta}$ is the average-twisted scalar curvature of any metric in $\beta$ on $B$, with twist  $\xi+\zeta$. Using that $p(\theta_t) =: \lambda_{E_t} \in C^{\infty}_{E_t}(X)$, we obtain
\begin{align*} 
-\frac{d}{dt} &\left( \scM_{\xi,k} (\phi_t ) \right) \\
=& \int_X \dot \phi_t \left( S(\omega_{t,b}) - \hat S_b + k^{-1} \left(\lambda_{R_t} + \lambda_{E_t} + S(\omega_B) - \Lambda_{\omega_B} (\xi + \zeta) - \hat S_{\xi+\zeta} \right) \right) \omega_t^m + O(k^{n-2}) \\
=& {m+n \choose n} k^{n-1} \int_X \dot \phi_t (\lambda_{R_t} +\lambda_{E_t}) \omega_t^{m}\wedge \omega_B^n + O(k^{n-2}) \\
=& {m+n \choose n} k^{n-1} \frac{d}{dt} \scN (\phi_t) + O(k^{n-2}) ,
\end{align*}
using that $\omega_{t}$ is relatively cscK, that $\omega_B$ solves the twisted equation $S(\omega_B) - \Lambda_{\omega_B} (\xi + \zeta) = \hat S_{\xi+\zeta,\beta}$, and that from Lemma \ref{lem:dotphi} and Corollary \ref{corollary-orthogonality} $\dot \varphi_t$ is orthogonal to $C^{\infty}_{R_t}(X)$, and finally that $p(\theta_t) = \lambda_{E_t} \in C^{\infty}_{E_t}(X)$ by definition. Integrating the above from zero to one gives the desired expansion.
\end{proof}

With this in place, we can prove Theorem \ref{thm:energy}.
\begin{proof} Suppose $\omega_X+ \ddb \psi$ is an optimal symplectic connection with isometry group $K$. Let $\tilde \omega_k \in k\beta + \alpha$ denote the twisted extremal metric constructed in Theorem \ref{thm:solutions}, which by construction are $K$-invariant. We use the reference metric $\omega_k = k \omega_B + \omega_X \in k\beta + \alpha$ in the definition of the modified twisted Mabuchi functional, and write $\tilde \omega_k = \omega_k + \ddb \tilde \psi_k$. We first show that we can approximate $\scM_{\xi,k} (\tilde \psi_k)$ by $\scM_{\xi,k} (\omega_k+ \ddb \psi)$, where $\omega_k = \omega_X + k \omega_B$. 

From Theorem \ref{thm:solutions}, we have $$\tilde \omega_k = \omega_X+\ddb \psi + k \omega_B + k^{-1} \ddb l_1 + O(k^{-2}),$$ for some $K$-invariant function $l_1 \in C^{\infty}_R(X)$. It follows that 
$$\scM_{\xi,k} ( \phi ) - \scM_{\xi,k} (\tilde \psi_k) = \scM_{\xi,k} ( \phi) - \scM_{\xi,k} (\psi + k^{-1} l_1) + O(k^{n-2}),$$
for any other $\omega_k + \ddb \phi \in k\beta + \alpha.$ Moreover, for notational simplicity denoting temporarily $\omega_{k, \psi + k^{-1}\ddb t l_1} =  k\omega_B +\omega_X + \ddb \psi +  t k^{-1} \ddb l_1$
\begin{align*} & \scM_{\xi,k} (\psi  ) -\scM_{\xi,k} (\psi +  k^{-1} l_1) \\
=& \int_0^1 \int_X k^{-1} l_1 \left( S(\omega_{k, \psi + k^{-1}\ddb t l_1}) - \Lambda_{\omega_{k, \psi + k^{-1}\ddb t l_1} } \xi  - \hat S_{k, \xi} \right) ( \omega_{k, \psi + k^{-1}\ddb t l_1}  )^{m+n} dt \\
&= O(k^{n-2}),
\end{align*}
since $\omega_X + \ddb \psi$ is relatively cscK and the holomorphy potential of the extremal vector field is $O(k^{-2})$. Thus $$\scM_{\xi,k} ( \phi ) - \scM_{\xi,k} ( \tilde \psi_k) = \scM_{\xi,k} (\phi) - \scM_{\xi,k} (\psi) + O(k^{n-2}),$$
for any $\phi$.

Next we use the asymptotic expansion of the the modified twisted Mabuchi functional established in Lemma \ref{lem:energyasymptotic}. By Theorem \ref{thm:lowerboundmabuchi}, the potential $\tilde \psi_k$ achieves the absolute minimum of this functional, so
$$\scM_{\xi} ( \phi) - \scM_{\xi,k} (\tilde \psi_k) \geq 0.$$
Thus  
\begin{align*} 0 &\leq \scM_{\xi,k} ( \phi ) - \scM_{\xi,k} (\tilde \psi_k) \\
&= \scM_{\xi,k} ( \phi ) - \scM_{\xi,k} (\psi) + O(k^{n-2}) \\
&= {m+n \choose n} k^{n-1} \left( \scN (\phi) - \scN (\psi) \right) + O(k^{n-2}),
\end{align*}
by Lemma \ref{lem:energyasymptotic}. It follows that $\scN (\phi) \geq \scN (\psi),$ which is what we wanted to prove.
\end{proof}

\begin{remark} From Remark \ref{rem:csckmabuchibound}, one obtains that if the twisted extremal metrics constructed on $X$ is actually twisted cscK, an optimal symplectic connection minimises the energy functional $\scN$ on the whole of $\scK_E$. Although this will seldom happen in practice, we believe the statement about  $\scN$ being bounded on $\scK_E$ to be true in general. 

The reason we require invariance is that the modified twisted  Mabuchi functional is only \emph{defined} on potentials invariant under the flow of the extremal vector field. As we have no control over the extremal vector field when constructing twisted extremal metrics on $X$, we are forced to work with potentials invariant under the maximal comapct subgroup $K=\Isom_0(\pi, \omega_X)$. We expect that a better understanding of geodesics in the space of relatively cscK metrics would allow us to remove our $K$-invariance assumption. We also expect, however, that boundedness on $K$-invariant metrics is equivalent to boundedness with no invariance assumption.
\end{remark}

Finally, the same proof as that of Theorem \ref{thm:energy}  allows us to deduce basic properties of the functional $\scN$ from corresponding properties of $\scM$. In order to state the third property, which demonstrates how $\scN$ depends on the reference metric, we will (from the second property) write $\scN_{\omega_X}(\omega_X')$ for value of the functional $\scN$ at $\omega_X'$ with reference metric $\omega_X$. 

\begin{lemma}\label{lemma:independence}The functional $\scN: \scK_E \to \R$ satisfies the following properties:
\begin{enumerate}
\item it is well-defined, independent of choice of path $\phi_t$. Moreover, $\scN$ is invariant under adding a constant to $\phi$;
\item it is independent of adding a constant to $\phi$, hence can viewed as a functional on the space of relatively cscK metrics, rather than potentials;
\item for relatively cscK metrics $\omega_X, \omega_X', \omega_X'' \in \alpha$, it satisfies the cocycle condition $$\scN_{\omega_X}(\omega_X') + \scN_{\omega_X'}(\omega_X'')+\scN_{\omega_X''}(\omega_X) = 0.$$
\end{enumerate}

 \end{lemma}

\begin{proof} (i) This follows from the proof of Theorem \ref{thm:energy} and the fact that the Mabuchi functional itself is well-defined \cite{mabuchi-functional}. Indeed, suppose $\phi_t, \psi_t$ are two paths of relatively cscK metrics joining $0$ to $\phi$. Then $k\omega_B + \omega_X + \ddb \phi_t$ and $k\omega_B + \omega_X + \ddb \psi_t$ are K\"ahler metrics in $k\beta + \alpha$ for all $k \gg 0$, and hence if $\scM_k$ denotes the Mabuchi functional (with respect to the reference metric $k\omega_B + \omega_X$), then $$\scM_k(\phi) = \int_0^1 \left(\frac{d}{dt}\scM_k(\phi_t)\right)dt =  \int_0^1 \left(\frac{d}{dt}\scM_k(\psi_t)\right)dt .$$ Expanding in $k$ and arguing as in the proof of Theorem \ref{thm:energy} shows that $$ \int_0^1\left(\frac{d}{dt}\scN(\phi_t)\right)dt =  \int_0^1 \left(\frac{d}{dt}\scN(\psi_t)\right)dt,$$ which implies that $\scN: \scK_E \to \R$ is well-defined, as required. 

(ii) This is obvious from the definition of $\scN$.

(iii) The Mabuchi functional satisfies the cocycle property \cite[Theorem 2.4]{mabuchi-functional} $$\scM_{k\omega_B +\omega_X}(k\omega_B +\omega_X') + \scM_{k\omega_B +\omega_X'}(k\omega_B +\omega_X'') + \scM_{k\omega_B +\omega_X ''}(k\omega_B +\omega_X) = 0,$$ and expanding in $k$ as above shows that the same property is true for the functional $\scN$.
\end{proof}

It follows from $(iii)$ that boundedness of $\scN$ is independent of choice of reference relatively cscK metric.

%\printbibliography

\vspace{4mm}


\begin{thebibliography}{99}

\bibitem{anchouche-biswas}
B. Anchouche and I. Biswas \emph{Einstein-Hermitian connections on polystable principal bundles over a compact K\"ahler manifold.} Amer. J. Math. 123 (2001), no. 2, 207--228.

\bibitem{berman-berndtsson}
R. Berman and B. Berndtsson \emph{Convexity of the $K$-energy on the space of K\"{a}hler metrics and uniqueness of extremal metrics.} J. Amer. Math. Soc. 30 (2017), no. 4, 1165--1196.

\bibitem{darvas}
T. Darvas \emph{The {M}abuchi geometry of finite energy classes} Adv. Math.  285 (2015), 182--219.

\bibitem{darvas-rubinstein}
T. Darvas and Y. A. Rubinstein \emph{Tian's properness conjectures and {F}insler geometry of the
              space of {K}\"{a}hler metrics} J. Amer. Math. Soc. 30 (2017), no. 2, 347--387.	
              
\bibitem{datar-szekelyhidi}
V. Datar and G. Sz\'ekelyhidi \emph{K\"ahler-Einstein metrics along the smooth continuity method} Geom. Funct. Anal. 26 (2016), no. 4, 975--1010. 	
              
\bibitem{uniform}  
R. Dervan \emph{Uniform stability of twisted constant scalar curvature K\"ahler metrics} Int. Math. Res. Not. IMRN 2016, no. 15, 4728--4783.               
              
\bibitem{moduli}
R. Dervan and P. Naumann \emph{Moduli of polarised manifolds via canonical K{\"a}hler metrics} (2018), arXiv:1810.02576

\bibitem{kahler}
R. Dervan and J. Ross \emph{K-stability for K\"ahler manifolds} Math. Res. Lett. 24 (2017), no. 3, 689--739. 

\bibitem{stablemaps}
R. Dervan and J. Ross \emph{Stable maps in higher dimensions} Math. Ann. 374 (2019), no. 3-4, 1033--1073. 

\bibitem{fibrations}
R. Dervan and L. M. Sektnan \emph{Extremal metrics of fibrations} Proc. Lond. Math. Soc. (3) 120 (2020), no. 4, 587--616.

\bibitem{morefibrations}
R. Dervan and L. M. Sektnan \emph{Optimal symplectic connections on holomorphic submersions} Comm. Pure. Appl. Math. (to appear)  (2019), arXiv:1907.11014

\bibitem{stablefibrations}
R. Dervan and L. M. Sektnan \emph{Moduli theory, stability of fibrations and optimal symplectic connections} 2019, arXiv:1911.12701              



\bibitem{donaldson-hermite-einstein}
S. K. Donaldson \emph{Anti self-dual Yang-Mills connections over complex algebraic surfaces and stable vector bundles.} Proc. London Math. Soc. (3) 50 (1985), no. 1, 1--26. 

\bibitem{donaldson-uniqueness-programme}
S. K. Donaldson \emph{Symmetric spaces, K\"ahler geometry and Hamiltonian dynamics} Northern California Symplectic Geometry Seminar, 13--33, Amer. Math. Soc. Transl. Ser. 2, 196, Adv. Math. Sci., 45, Amer. Math. Soc., Providence, RI, (1999). 

\bibitem{donaldson-scalar-curvature}
S. K. Donaldson \emph{Scalar curvature and projective embeddings. I} J. Differential Geom. 59 (2001), no. 3, 479--522. 

\bibitem{fine1}
J. Fine \emph{Constant scalar curvature K\"ahler metrics on fibred complex surfaces} J. Differential Geom. 68 (2004), no. 3, 397--432.

\bibitem{fine2}
J. Fine \emph{Fibrations with constant scalar curvature K\"ahler metrics and the CM-line bundle}
Math. Res. Lett. 14 (2007), no. 2, 239--247.

\bibitem{fujiki-schumacher}
A. Fujiki and G. Schumacher \emph{The moduli space of extremal compact K\"ahler manifolds and generalized Weil-Petersson metrics} Publ. Res. Inst. Math. Sci. 26 (1990), no. 1, 101--183.

\bibitem{futaki-mabuchi}
A. Futaki and T. Mabuchi \emph{Bilinear forms and extremal Kähler vector fields associated with Kähler classes.} Math. Ann. 301 (1995), no. 2, 199--210. 



\bibitem{gauduchon}
P. Gauduchon \emph{Calabi’s extremal K\"ahler metrics: An elementary introduction} Book available online.

\bibitem{hashimoto}
Y. Hashimoto \emph{Existence of twisted constant scalar curvature K\"ahler metrics with a large twist} Math. Z. 292 (2019), no. 3-4, 791--803.

\bibitem{hong}
Y, J. Hong \emph{Constant Hermitian scalar curvature equations on ruled manifolds} J. Differential Geom. 53 (1999), no. 3, 465--516. 


\bibitem{inoue}
E. Inoue \emph{The moduli space of Fano manifolds with K\"ahler-Ricci solitons} Adv. Math. no. 357 (2019). 1--65.

\bibitem{keller}
J. Keller \emph{Twisted balanced metrics} Lie Groups:  New research,  Mathematics Research Developments (2009), 267--281.

\bibitem{mabuchi-functional} 
T. Mabuchi \emph{K-energy maps integrating Futaki invariants} Tohoku Math. J. (2) 38 (1986), no. 4, 575--593. 

\bibitem{mabuchi}
T. Mabuchi \emph{Some symplectic geometry on compact K\"ahler manifolds. I} Osaka J. Math. 24 (1987), no. 2, 227--252. 

\bibitem{mcduff-salamon} 
D. McDuff and D. Salamon \emph{Introduction to symplectic topology}, third ed., Oxford Graduate Texts in Mathematics, Oxford University Press, Oxford (2017).  xi+623 pp.

\bibitem{michor}
P. Michor \emph{Topics in differential geometry} Graduate Studies in Mathematics, vol. 93, American Mathematical Society, Providence, RI (2008). xii+494 pp.

\bibitem{gabor-book}
G. Sz\'ekelyhidi \emph{An introduction to extremal K\"ahler metrics}  Graduate Studies in Mathematics, 152. American Mathematical Society, Providence, RI (2014). xvi+192 pp.

\bibitem{tian-book} G. Tian \emph{Canonical metrics in K\"ahler geometry} Lectures in Mathematics ETH Z\"urich, Birkh\"auser Verlag, Basel (2000). Notes taken by Meike Akveld. vi+101 pp.

\end{thebibliography}
\end{document}